\documentclass[a4paper]{amsart}
\usepackage{amsmath,amsthm,amssymb,latexsym,epic,bbm,comment,mathbbol}
\usepackage{graphicx,enumerate,stmaryrd,color}
\usepackage[all,2cell]{xy}
\xyoption{2cell}

\newtheorem{theorem}{Theorem}
\newtheorem{lemma}[theorem]{Lemma}
\newtheorem{corollary}[theorem]{Corollary}
\newtheorem{proposition}[theorem]{Proposition}

\usepackage[all]{xy}
\usepackage[active]{srcltx}
\usepackage[parfill]{parskip}
\usepackage{enumerate}

\newcommand{\tto}{\twoheadrightarrow}
\font\sc=rsfs10
\newcommand{\cC}{\sc\mbox{C}\hspace{1.0pt}}
\newcommand{\cG}{\sc\mbox{G}\hspace{1.0pt}}

\newcommand{\cI}{\sc\mbox{I}\hspace{1.0pt}}
\newcommand{\cJ}{\sc\mbox{J}\hspace{1.0pt}}
\newcommand{\cS}{\sc\mbox{S}\hspace{1.0pt}}

\newcommand{\cA}{\sc\mbox{A}\hspace{1.0pt}}

\newcommand{\cQ}{\sc\mbox{Q}\hspace{1.0pt}}

\font\scc=rsfs7

\newcommand{\ccQ}{\scc\mbox{Q}\hspace{1.0pt}}

\begin{document}

\title[Simple transitive 2-representations for Soergel bimodules]
{Simple transitive $2$-representations for some $2$-subcategories of Soergel bimodules}
\author{Marco Mackaay and Volodymyr Mazorchuk}

\begin{abstract}
We classify simple transitive $2$-representations of certain $2$-sub\-ca\-te\-go\-ri\-es of the
$2$-category of Soergel bimodules over the coinvariant algebra in Coxeter types
$B_2$ and $I_2(5)$. In the $I_2(5)$ case it turns out that simple transitive $2$-representations
are exhausted by cell $2$-representations. In the $B_2$ case we show that, apart from cell 
$2$-representations, there is a unique, up to equivalence, additional simple transitive 
$2$-representation and we give an explicit construction of this $2$-representation.
\end{abstract}

\maketitle

\section{Introduction and description of the results}\label{s1}

Classification problems are, historically, the main driving force of representation theory. 
The desire to understand and, in particular, classify certain classes of representations
of a given group or algebra was behind the majority of research in the general area of
representation theory since its birth.

The abstract $2$-representation theory, which originated in \cite{BFK,CR,KL,Ro}, studies functorial 
actions of $2$-categories. The ``finite-dimensional'' part of this theory, that is 
$2$-representation theory of finitary $2$-categories, was systematically developed in the
series \cite{MM1,MM2,MM3,MM4,MM5,MM6} and continued in \cite{Xa,Zh1,Zh2}. In particular,
the paper \cite{MM5} proposes a very good candidate for the notion of a ``simple'' 
$2$-representation, called a {\em simple transitive} $2$-representation. In the same paper
one finds a classification of such $2$-representations for a special class of finitary
$2$-categories with involution which includes the $2$-category of Soergel bimodules 
over the coinvariant algebra of the symmetric group. This is extended in \cite{MM6}
to more general $2$-categories and (slightly) more general classes of $2$-representations.
Some nice applications of these classification results were obtained in \cite{KM}.

The classification of simple transitive $2$-representations for some ``smallest'' 
$2$-ca\-te\-go\-ri\-es which do not fit the setup and methods of \cite{MM5,MM6} was 
completed in \cite{MZ,Zi}. All the results mentioned above have, however, one common
feature. It turns out that in all these cases the simple transitive $2$-representations
are exhausted by the so-called {\em cell $2$-representations} defined and studied in \cite{MM1}.
So far there was only one, quite artificial, example of a family of simple transitive $2$-representations
which are not equivalent to cell $2$-representations, constructed in \cite[Subsection~3.2]{MM5}
using transitive group actions.

In the present paper we study simple transitive $2$-representations of a
certain $2$-sub\-quo\-tient $\cQ_n$ of the
$2$-category of Soergel bimodules for the dihedral group $D_{2\cdot n}$, where $n\geq 3$.
For $n=3$, this $2$-category fits into the setup of \cite{MM5} and hence the classification 
result from \cite{MM5} directly applies. For $n>3$, the $2$-category $\cQ_n$ must be studied
by other methods. 

We show that every simple transitive $2$-representation of  $\cQ_5$ is equivalent to a
cell $2$-rep\-re\-sentation, see Theorem~\ref{thm11}. We also show that, apart from cell 
$2$-representations, there is a unique, up to equivalence, simple transitive 
$2$-representation of $\cQ_4$, see Theorem~\ref{thm21}. The corresponding $2$-representation is
explicitly constructed in Subsection~\ref{s5.3}. This subsection is the heart of this paper.
Construction of this new $2$-representation is based on the careful interplay of several 
category-theoretic tricks. The case $n>5$ seems, at the moment, computationally too difficult.

The paper is organized as follows: in Section~\ref{s2} we collect all necessary preliminaries
from $2$-representation theory. In Section~\ref{s3} we recall the definition and combinatorics
of the $2$-category of Soergel bimodules over a dihedral group and define the $2$-category
$\cQ_n$, our main object of study. Theorem~\ref{thm11} is proved in Section~\ref{s4}.
Theorem~\ref{thm21} is proved in Section~\ref{s5}.
\vspace{3mm}

{\bf Acknowledgment.} A major part of the research which led to this paper was done during the visit of the
first author to Uppsala University which was partially supported by the Swedish Research Council.
Hospitality of Uppsala University and financial support of the Swedish Research Council are
gratefully acknowledged. We thank Aaron Chan for comments on the original version.
We are grateful to the referee for numerous comments and suggestions which helped
to improve exposition.

\section{Generalities on $2$-categories and $2$-representations}\label{s2}

\subsection{Notation and conventions}\label{s2.1}

We work over $\mathbb{C}$ and write $\otimes$ for $\otimes_{\mathbb{C}}$. 
By a module we mean a {\em left} module. Maps are composed from right to left.

\subsection{Finitary and fiat $2$-categories}\label{s2.2}

We refer the reader to \cite{Le,Mc,Ma} for generalities on $2$-categories. 

A $2$-category is a category enriched over the monoidal category $\mathbf{Cat}$ of small categories.
Thus, a $2$-category $\cC$ consists of objects (denoted by Roman lower case letters in 
a typewriter font), $1$-morphisms (denoted by capital Roman letters), and  $2$-morphisms 
(denoted by Greek lower case letters), composition of  $1$-morphisms, horizontal and 
vertical compositions of $2$-morphisms (denoted $\circ_0$ and $\circ_1$ respectively), 
identity $1$-morphisms and identity $2$-morphisms. These must satisfy the obvious collection 
of axioms. For a $1$-morphism $\mathrm{F}$, we denote by $\mathrm{id}_{\mathrm{F}}$ the corresponding
identity $2$-morphism. As usual, we often write $\mathrm{F}(\alpha)$ for $\mathrm{id}_{\mathrm{F}}\circ_0\alpha$
and $\alpha_{\mathrm{F}}$ for $\alpha\circ_0\mathrm{id}_{\mathrm{F}}$.

A $2$-category $\cC$ is called {\em finitary} if, for each pair $(\mathtt{i},\mathtt{j})$
of objects in $\cC$, the category  $\cC(\mathtt{i},\mathtt{j})$ is 
an idempotent split, additive and Krull-Schmidt $\mathbb{C}$-linear category with finitely many
isomorphism classes of indecomposable objects and finite dimensional morphism spaces; moreover, 
all compositions must be compatible with these additional structures, see \cite{MM1} for details.

A finitary $2$-category $\cC$ is called {\em fiat} if it has a weak involution $\star$ together with
adjunction $2$-morphisms satisfying the usual axioms of adjoint functors, for each pair 
$(\mathrm{F},\mathrm{F}^{\star})$ of $1$-morphisms, see \cite{MM1} for details.

\subsection{$2$-representations}\label{s2.3}

For a finitary $2$-category $\cC$, we consider the $2$-category $\cC$-afmod of all 
{\em finitary $2$-representations} of $\cC$ as defined in \cite{MM3}. Objects of
$\cC$-afmod are strict functorial actions on idempotent split, additive and Krull-Schmidt 
$\mathbb{C}$-linear categories with finitely many isomorphism classes of indecomposable objects 
and finite dimensional morphism spaces. Furthermore, $1$-morphisms in $\cC$-afmod are strong 
$2$-natural transformations and $2$-morphisms are modifications.

Similarly we can consider the $2$-category $\cC$-mod of all 
{\em abelian $2$-representations} of $\cC$, that is functorial actions on categories
equivalent to module categories over finite dimensional algebras, see  \cite{MM3} for details. 
We have the diagrammatically defined abelianization $2$-functor
\begin{displaymath}
\overline{\hspace{1mm}\cdot\hspace{1mm}}:
\cC_A\text{-}\mathrm{afmod}\to \cC_A\text{-}\mathrm{mod},
\end{displaymath}
see \cite[Subsection~4.2]{MM2} for more details.

Two $2$-representations are said to be {\em equivalent} if there is a strong $2$-natural 
transformation $\Phi$ between them such that the restriction of $\Phi$ to each object 
in $\cC$ is an equivalence of categories.

A finitary $2$-representation $\mathbf{M}$ of $\cC$ will 
be called {\em transitive} provided that, for each indecomposable objects 
$X$ and $Y$ in $\displaystyle \coprod_{\mathtt{i}}\mathbf{M}(\mathtt{i})$, there is a $1$-morphism
$\mathrm{F}$ in $\cC$ such that $Y$ is isomorphic to a direct summand of 
the object $\mathbf{M}(\mathrm{F})\, X$. A transitive $2$-representation $\mathbf{M}$ 
is called {\em simple transitive } if $\displaystyle \coprod_{\mathtt{i}}\mathbf{M}(\mathtt{i})$ 
does not have any non-zero proper $\cC$-invariant ideals.

Similarly one can define the notion of {\em transitive (based) module} over any 
positively based algebra $A$ in the sense of \cite[Section~9]{KM2}. Here transitivity means 
that the basis of $V$ has the property that, for any elements $v$ and $w$ in this basis, 
there is an element $a$ in the basis of $A$ such that $v$ appears with a non-zero coefficient in $aw$.

For simplicity, we often use the module notation $\mathrm{F}\, X$ instead of the
corresponding representation notation $\mathbf{M}(\mathrm{F})\, X$.

\subsection{Combinatorics}\label{s2.4}

For a finitary $2$-category $\cC$, we denote by $\mathcal{S}[\cC]$ the corresponding 
{\em multisemigroup} as defined in \cite[Section~3]{MM2}. By $\leq_L$, $\leq_R$ and $\leq_J$
we denote the corresponding {\em left}, {\em right} and {\em two-sided} orders on
$\mathcal{S}[\cC]$. Equivalence classes for these orders are called {\em cells}.
For simplicity, we will abuse the language and say ``cells of $\cC$'' instead of 
``cells of $\mathcal{S}[\cC]$''.

If $\mathcal{J}$ is a two-sided cell in $\mathcal{S}[\cC]$, then the 
$2$-category $\cC$ is called {\em $\mathcal{J}$-simple} provided that 
any non-zero two-sided $2$-ideal of $\cC$ contains the identity $2$-morphisms for all 
$1$-morphisms in $\mathcal{J}$, see \cite{MM2}. 

\subsection{Cell $2$-representations}\label{s2.5}

For $\mathtt{i}\in \cC$, we denote by $\mathbf{P}_{\mathtt{i}}:=\cC_A(\mathtt{i},{}_-)$
the corresponding {\em principal} $2$-representation. For a left cell $\mathcal{L}$
in $\mathcal{S}[\cC]$, we denote by $\mathbf{C}_{\mathcal{L}}$ the corresponding 
{\em cell} $2$-representation, see \cite{MM1,MM2} for details.

\subsection{Matrices in the Grothendieck group}\label{s0-1.6}

For a finitary $2$-category and a finitary $2$-representation $\mathbf{M}$ of $\cC$,
let us fix a complete and irredundant list of representatives of isomorphism classes of 
indecomposable objects in $\displaystyle \coprod_{\mathtt{i}}\mathbf{M}(\mathtt{i})$.
Then, for any $1$-morphism $\mathrm{F}$, we have the corresponding matrix $\Lparen \mathrm{F}\Rparen$
which counts multiplicities in direct sum decompositions of the images of indecomposable
objects under $\mathrm{F}$.

If $\cC$ is fiat, then each $\overline{\mathbf{M}}(\mathrm{F})$ is exact and we have the
matrix $\llbracket \mathrm{F}\rrbracket$ which counts composition multiplicities 
of the images of simple objects under $\mathrm{F}$. By adjunction, the matrix $\Lparen \mathrm{F}\Rparen$
is transposed to the matrix $\llbracket \mathrm{F}^{\star}\rrbracket$.

\section{The $2$-category $\cQ_n$}\label{s3}

\subsection{Soergel bimodules for dihedral groups}\label{s3.1}

We refer the reader to \cite{So,So2,E,EW} for more information and details on Soergel bimodules.

For $n\geq 3$, consider the dihedral group $D_{2\cdot n}$ of symmetries of a regular $n$-gon in $\mathbb{R}^2$
with its corresponding {\em defining} module $\mathbb{C}^2$. The group $D_{2\cdot n}$ has Coxeter presentation

\begin{displaymath}
D_{2\cdot n}=\langle s,t\,:\, s^2=t^2=(st)^n=e\rangle. 
\end{displaymath}
We may assume that, in the defining representation, the elements $s$ and $t$ act via the matrices
\begin{displaymath}
\left(\begin{array}{cc}1&0\\0&-1\end{array}\right)\quad\text{ and } \quad
\left(\begin{array}{cc}\cos(\frac{2\pi}{n})&\sin(\frac{2\pi}{n})\\
\sin(\frac{2\pi}{n})&-\cos(\frac{2\pi}{n})\end{array}\right),
\end{displaymath}
respectively.
Let $\leq$ denote the Bruhat order on $D_{2\cdot n}$. For $w\in D_{2\cdot n}$, set
\begin{displaymath}
\underline{w}:=\sum_{v\leq w}v.
\end{displaymath}
Then $\{\underline{w}\,:\,w\in D_{2\cdot n}\}$ is the {\em Kazhdan-Lusztig} basis of $\mathbb{Z}D_{2\cdot n}$.
We denote by $w_0$ the longest element in $D_{2\cdot n}$.

Let $\mathtt{C}$ be the coinvariant algebra associated to the defining $D_{2\cdot n}$-module $\mathbb{C}^2$.
Let, further, $\mathtt{C}^s$ denote the subalgebra of 
$s$-invariants in $\mathtt{C}$ and $\mathtt{C}^t$ denote the subalgebra of 
$t$-invariants in $\mathtt{C}$. A {\em Soergel $\mathtt{C}\text{-}\mathtt{C}$--bimodule} is a 
$\mathtt{C}\text{-}\mathtt{C}$--bimodule isomorphic to a bimodule from the additive closure of the
monoidal category of  $\mathtt{C}\text{-}\mathtt{C}$--bimodules generated by 
\begin{displaymath}
\mathtt{C}\otimes_{\mathtt{C}^s}\mathtt{C} \qquad\text{ and }\qquad
\mathtt{C}\otimes_{\mathtt{C}^t}\mathtt{C}. 
\end{displaymath}
Isomorphism classes of indecomposable Soergel bimodules are naturally indexed by $w\in D_{2\cdot n}$
and we denote by $B_w$ a fixed representative from such a class. 

Consider a small category $\mathcal{C}$ equivalent to $\mathtt{C}\text{-}\mathrm{mod}$. Define
the $2$-category $\cS_n$ of {\em Soergel bimodules} (associated to $\mathcal{C}$) as follows:
\begin{itemize}
\item $\cS_n$ has one object $\mathtt{i}$, which we can identify with $\mathcal{C}$;
\item $1$-morphisms in $\cS_n$ are all endofunctors of $\mathcal{C}$ which are isomorphic to endofunctors
given by tensoring with Soergel $\mathtt{C}\text{-}\mathtt{C}$--bimodules;
\item $2$-morphisms in $\cS_n$ are natural transformations of functors
(these correspond to homomorphisms of Soergel $\mathtt{C}\text{-}\mathtt{C}$--bimodules).
\end{itemize}
The $2$-category $\cS_n$ is fiat.

For $w\in D_{2\cdot n}$, let $\theta_w$ denote a fixed representative in the isomorphism class of 
indecomposable $1$-morphisms given by tensoring with $B_w$. The $2$-category $\cS_n$ has three
two-sided cells

\begin{displaymath}
\mathcal{J}_e:=\{\theta_e\},\quad \mathcal{J}_{w_0}:=\{\theta_{w_0}\},\quad  
\mathcal{J}_s:=\{\theta_w\,:\, w\neq e,w_0\},
\end{displaymath}

which are linearly ordered $\mathcal{J}_e\leq_J\mathcal{J}_s\leq_J\mathcal{J}_{w_0}$.
We also have four left cells 
\begin{gather*}
\mathcal{L}_e:=\{\theta_e\},\quad \mathcal{L}_{w_0}:=\{\theta_{w_0}\},\quad  
\mathcal{L}_s:=\{\theta_w\,:\, w\neq w_0,\,\, ws<w\},\\
\mathcal{L}_t:=\{\theta_w\,:\, w\neq w_0,\,\, wt<w\}.
\end{gather*}

The left order on these left cells is given by 
\begin{displaymath}
\mathcal{L}_e\leq_L\mathcal{L}_s\leq_L\mathcal{L}_{w_0} \quad\text{ and }\quad
\mathcal{L}_e\leq_L\mathcal{L}_t\leq_L\mathcal{L}_{w_0},
\end{displaymath}
with $\mathcal{L}_s$ and $\mathcal{L}_t$ being incomparable.
Applying $w\mapsto w^{-1}$, one obtains right cells 
\begin{gather*}
\mathcal{R}_e:=\{\theta_e\},\quad \mathcal{R}_{w_0}:=\{\theta_{w_0}\},\quad  
\mathcal{R}_s:=\{\theta_w\,:\, w\neq w_0,\,\, sw<w\},\\
\mathcal{R}_t:=\{\theta_w\,:\, w\neq w_0,\,\, tw<w\}.
\end{gather*}
and the right order
\begin{displaymath}
\mathcal{R}_e\leq_R\mathcal{R}_s\leq_R\mathcal{R}_{w_0} \quad\text{ and }\quad
\mathcal{R}_e\leq_R\mathcal{R}_t\leq_R\mathcal{R}_{w_0},
\end{displaymath}
with $\mathcal{R}_s$ and $\mathcal{R}_t$ being incomparable.

For the decategorification $[\cS_n]$, we have an isomorphism
\begin{displaymath}
[\cS_n](\mathtt{i},\mathtt{i}) \cong \mathbb{Z}D_{2\cdot n}
\end{displaymath}
which sends $[\theta_w]$ to $\underline{w}$, see \cite{So2} (and also
\cite{E}).

\subsection{The $2$-subquotient $\cQ_n$ of $\cS_n$}\label{s3.2}

Let $\cI$ be the $2$-ideal of $\cS_n$ generated by $\mathrm{id}_{\theta_{w_0}}$. Denote by 
$\cQ'_n$ the $2$-subcategory of $\cS_n/\cI$ defined as follows:
\begin{itemize}
\item $\cQ'_n$ has the same objects as $\cS_n/\cI$;
\item $1$-morphisms in $\cQ'_n$ are all $1$-morphisms in $\cS_n/\cI$ which are isomorphic to 
direct sums of the  $1$-morphisms $\theta_e$ and $\theta_w\in \mathcal{L}_s\cap \mathcal{R}_s$;
\item $2$-morphisms in $\cQ'_n$ are inherited from $\cS_n/\cI$.
\end{itemize}

The $2$-category $\cQ'_n$ inherits from $\cS_n$ the structure of a fiat $2$-category.
Abusing notation, we will denote indecomposable $1$-morphisms in $\cQ'_n$ by the same
symbols as for $\cS_n$. Directly from the construction it follows that $\cQ'_n$ has
two two-sided cells which are, at the same time, both left and right cells:
\begin{displaymath}
\mathcal{J}_e=\mathcal{L}_e=\mathcal{R}_e=\{\theta_e\},\qquad
\mathcal{J}_s=\mathcal{L}_s=\mathcal{R}_s=\{\theta_w\,:\,w\neq w_0;\,\, ws<w,\,\, sw<w\}
\end{displaymath}
such that $\mathcal{J}_e$ is the minimum element with respect to the left, right and two-sided 
orders and $\mathcal{J}:=\mathcal{J}_s$ is the maximum element with respect to the left, 
right and two-sided  orders.

We define the $2$-category $\cQ_n$ as the quotient of $\cQ'_n$ by the unique $2$-ideal
which is maximal in the set of all $2$-ideals which do not contain any 
$\mathrm{id}_{\mathrm{F}}$, for non-zero $\mathrm{F}$. This ideal exists, see  e.g.
\cite[Lemma~4]{MM5}, however, we do not know any explicit set of generators for it. 
The $2$-category $\cQ_n$ is the main
object of study in the present paper. By construction, the $2$-category $\cQ_n$ is
fiat and $\mathcal{J}$-simple.

We also denote by $\cA_n$ the $2$-full $2$-subcategory of $\cQ_n$ with $1$-morphisms in 
the additive closure of $\theta_e$ and $\theta_s$. 

\subsection{Simple transitive $2$-representations of $\cQ_3$}\label{s3.3}

For $n=3$, the two-sided cell $\mathcal{J}_s$ contains only one element. Therefore
$\cQ_3=\cA_3$ fits into the general setup of \cite{MM5}. Consequently, by \cite[Theorem~18]{MM5},
each simple transitive $2$-representation of $\cQ_3$ is equivalent to a cell 
$2$-representation. In fact, from \cite[Theorem~13]{MM3} it follows that 
the $2$-category $\cQ_3$ is biequivalent to the 
$2$-category of Soergel bimodules for the symmetric group $S_2$.

For the cell $\mathcal{L}_e$, the corresponding cell $2$-representation is, roughly speaking, 
given by the obvious functorial action on $\mathbb{C}$-mod, where $\theta_s$ acts as the zero functor.

For the cell $\mathcal{L}_s$, the corresponding cell $2$-representation is, roughly speaking, given by
a functorial action of $\cQ_3$ on the category of projective modules over the algebra 
$D=\mathbb{C}[x]/(x^2)$ of dual numbers. Here $\theta_s$ acts via tensoring with the 
projective $D$-$D$--bimodules $D\otimes D$. This is a special case of a very general
picture which we describe in the next subsection.

\subsection{Some $2$-categories similar to $\cQ_3$}\label{s3.4}

Let $A$ be a finite dimensional algebra and $B$ a subalgebra of $A$. Assume the following:
\begin{enumerate}[$($I$)$]
\item\label{e1} $A$ is local, commutative and Frobenius;
\item\label{e2} $B$ is Frobenius;
\item\label{e3} $A$ is free of rank two, as a $B$-module.
\end{enumerate}
Let $\mathcal{A}$ be a small category equivalent to $A$-mod.
Let $\cG$ denote the $2$-category defined as follows:
\begin{itemize}
\item $\cG$ has one object $\mathtt{i}$, which we identify with $A$-mod;
\item $1$-morphisms in $\cG$ are endofunctors in $A$-mod in the additive closure of
the identity functor and the functor of tensoring with $A\otimes_B A$;
\item $2$-morphisms in $\cG$ are natural transformations of functors.
\end{itemize}

This is well-defined because of the isomorphism
\begin{equation}\label{eqnn5}
A\otimes_B A\otimes_A A\otimes_B A \cong
\big(A\otimes_B A\big) \oplus \big(A\otimes_B A\big)
\end{equation}
which follows from \eqref{e3}.

We denote by $\mathrm{F}$ a $1$-morphism in $\cG$ given by tensoring with $A\otimes_B A$. 
Then we have $\mathrm{F}^2\cong \mathrm{F}\oplus \mathrm{F}$ from \eqref{eqnn5}.

\begin{lemma}\label{lem55}
The $2$-category $\cG$ is fiat. 
\end{lemma}

\begin{proof}
The computation in the proof of \cite[Proposition~24]{MM6} shows that there is an 
isomorphism of $A$-$A$--bimodules as follows:
\begin{displaymath}
\mathrm{Hom}_{\mathbb{C}}(A\otimes_B A,\mathbb{C})\cong A\otimes_B A.
\end{displaymath}
Therefore we can define the weak involution $\star$ using $\mathrm{Hom}_{\mathbb{C}}({}_-,\mathbb{C})$.
We have $\mathrm{F}^{\star}\cong\mathrm{F}$.

To prove existence of adjunction morphisms,  it is enough to show that 
$\mathrm{F}$ is self-adjoint. Note that $\mathrm{F}$ is the composition of the restriction from $A$ to $B$
followed by induction from $B$ to $A$. As both are symmetric, we compute the right
adjoint of the restriction:
\begin{displaymath}
\begin{array}{rcl} 
\mathrm{Hom}_{B}({}_BA_A,B)&\cong& \mathrm{Hom}_{B}({}_BA_A,\mathrm{Hom}_{\mathbb{C}}(B,\mathbb{C}))\\
&\cong& \mathrm{Hom}_{\mathbb{C}}(B\otimes_B{}_BA_A,\mathbb{C})\\
&\cong& \mathrm{Hom}_{\mathbb{C}}({}_BA_A,\mathbb{C})\\
&\cong& {}_AA_B,
\end{array}
\end{displaymath}
where we used $\mathrm{Hom}_{\mathbb{C}}(A,\mathbb{C}))\cong A$
as $A$ is symmetric and $\mathrm{Hom}_{\mathbb{C}}(B,\mathbb{C}))\cong B$
as $B$ is symmetric. This shows that  restriction is biadjoint to induction. Therefore
$\mathrm{F}$ is self-adjoint. 
\end{proof}

Recall also the $2$-category $\cC_D$ from \cite[Subsection~7.3]{MM1}, 
defined as follows: let $\mathcal{D}$ be a small category equivalent to $D$-mod.
\begin{itemize}
\item $\cC_D$ has one object $\mathtt{i}$, which we identify with  ;
\item $1$-morphisms in $\cC_D$ are endofunctors of $\mathcal{D}$ in the additive closure 
of the identity functor and the functor of tensoring with $D\otimes_{\mathbb{C}} D$;
\item $2$-morphisms are natural transformations of functors.
\end{itemize}

\begin{proposition}\label{propn51}
The $2$-category $\cG$ has a unique ideal $\cJ$ which is maximal in the set of all ideals
that do not contain the identity $2$-morphisms of non-zero one-morphisms. The quotient
$\cG/\cJ$ is biequivalent to $\cC_D$. 
\end{proposition}
Our description of the cell $2$-representation $\mathbf{C}_{\mathcal{L}_s}$
for the  $2$-category $\cQ_3$ follows directly from Proposition~\ref{propn51}.

\begin{lemma}\label{lemnn52}
Let $L\in\mathcal{A}$ be a simple, then $\mathrm{F}(L)$ is indecomposable of length two. 
\end{lemma}

\begin{proof}
As both $A$ and $B$ are local and commutative, see  \eqref{e1}, they both are basic and have a unique
simple module. Therefore, the simple $A$-module $L$ is $1$-dimensional and restricts to a 
simple $B$-module, which we call $L'$. By adjunction
\begin{displaymath}
\mathrm{Hom}_{A}(A\otimes_B L',L)\cong \mathrm{Hom}_B(L',L') 
\end{displaymath}
and the latter space is $1$-dimensional as $L'$ is $1$-dimensional. Therefore the module
$A\otimes_B L'\cong \mathrm{F}(L)$ has simple top. The module $\mathrm{F}(L)$ has length two as
$A$ is free over $B$ of rank two. The claim follows.
\end{proof}

Note that $\mathrm{End}_A(\mathrm{F}(L))\cong D$, where $x\in D$ corresponds to the nilpotent 
endomorphism of $\mathrm{F}(L)$ which maps the top of 
$\mathrm{F}(L)$ to the socle of $\mathrm{F}(L)$. 
This brings $D$ into the picture.  Now we can prove Proposition~\ref{propn51}.

\begin{proof}[Proof of Proposition~\ref{propn51}.]
Consider the additive closure $\mathcal{G}$ of $M=\mathrm{F}(L)$. From \eqref{eqnn5} it follows that
the additive category $\mathcal{G}$ is invariant under the action of $\cG$.
Let $\mathbf{M}$ denote the restriction $2$-representation of $\cG$ corresponding to
the action of $\cG$ on $\mathcal{G}$. Note that $\mathcal{G}$ is equivalent to the category $D$-proj
of projective $D$-modules. The abelianization of $\mathcal{G}$ is equivalent to the additive 
subcategory generated by $L$ and $M$, via the functor of taking a projective presentation
(we note that the additive  subcategory generated by $L$ and $M$ is, in fact, abelian).

Now, $\mathrm{F}$ is a self-adjoint functor which maps a simple object $L$ to 
a projective object $M$ (in the abelianization of $\mathcal{G}$).
By \cite[Lemma~13]{MM5}, $\mathrm{F}$ is given by tensoring with $D\otimes_{\mathbb{C}}D$, via 
an equivalence of $\mathcal{G}$ with $D$-proj. Therefore the representation map $\mathbf{M}$ defines a 
$2$-functor $\Phi$ from $\cG$ to a $2$-category biequivalent to $\cC_D$,
via an equivalence of $\mathcal{G}$ with $D$-proj. 
Moreover, from the classification of  simple fiat $2$-categories in \cite[Theorem~13]{MM3},
we have that the image of $\Phi$ is, in fact, biequivalent to $\cC_D$. 

The kernel of $\Phi$ is contained in $\cJ$, by construction. Since $\cC_D$ is $\mathcal{J}$-simple,
where $\mathcal{J}=\{D\otimes_{\mathbb{C}}D\}$, see \cite{MM1,MM3}, we have that the kernel of
$\Phi$ equals $\cJ$ and the claim of the proposition follows. 
\end{proof}

\begin{corollary}\label{cornew01}
Let $\mathbf{M}\in \cG$-afmod and $L\in \overline{\mathbf{M}}(\mathtt{i})$ be simple. Then 
$\mathrm{F}(L)$ is either zero or indecomposable. In case  $\mathrm{F}(L)$ is indecomposable,
then both the top and the socle of $\mathrm{F}(L)$ are isomorphic to $L$, moreover,
$\mathrm{F}(\mathrm{Rad}(L)/\mathrm{Soc}(L))=0$.
\end{corollary}

\begin{proof}
As $\cG$ is fiat with strongly regular two-sided cells, every simple transitive 
$2$-representation of $\cG$ is a cell $2$-representation by \cite[Theorem~15]{MM5}.
Hence there are exactly two different simple transitive 
$2$-representation of $\cG$ and the matrix $\Lparen \mathrm{F}\Rparen$ for these
two $2$-representation is $(0)$ or $(2)$. 

From \cite[Theorem~25]{ChMa} it follows that, for any $\mathbf{M}$, the matrix 
$\Lparen \mathrm{F}\Rparen$ has the form
\begin{displaymath}
\left(\begin{array}{cc}2E&*\\0&0\end{array}\right), 
\end{displaymath}
where $E$ is the identity matrix, up to permutation of basis elements. By adjunction, this implies 
$\mathrm{F}(L)=0$ or $[\mathrm{F}(L):L]=2$. In the latter case exactness of $\mathrm{F}$
yields that $\mathrm{F}$ annihilates all simple subquotients of $\mathrm{F}(L)$ which 
are not isomorphic to $L$. Therefore, by adjunction, only $L$ can appear in the
top or socle of $\mathrm{F}(L)$. Now, from indecomposability we get that both socle
and top must be isomorphic to $L$. The claim follows.
\end{proof}

\section{Simple transitive $2$-representations of $\cQ_5$}\label{s4}

\subsection{The result}\label{s4.1}

Our main result for the $2$-category $\cQ_5$ is the following:

\begin{theorem}\label{thm11}
Every simple transitive $2$-representation of $\cQ_5$ is equivalent to a
cell $2$-representation.
\end{theorem}

\subsection{The algebra $[\cQ_5](\mathtt{i},\mathtt{i})$}\label{s4.2}

We identify the ring $R=[\cQ_5](\mathtt{i},\mathtt{i})$ with the corresponding subquotient of 
$\mathbb{Z}D_{2\cdot 5}$.

With this identification, the ring $R$ has basis 
$\{\underline{e},\underline{s},\underline{sts}\}$ and the following multiplication table
of $x\cdot y$:

\begin{equation}\label{eq0}
\begin{array}{c||c|c|c}
x\setminus y& \underline{e} & \underline{s} & \underline{sts} \\
\hline\hline
\underline{e}& \underline{e} & \underline{s} & \underline{sts} \\
\hline
\underline{s}& \underline{s} & 2\underline{s} & 2\underline{sts} \\
\hline
\underline{sts}& \underline{sts} & 2\underline{sts} & 2\underline{s}+2\underline{sts} 
\end{array}
\end{equation}
The $\mathbb{C}$-algebra $A:=\mathbb{C}\otimes_{\mathbb{Z}}R$ is commutative and split semisimple.


The three $1$-dimensional representations of $A$ are given by the following table which describes
the action of the basis elements in any fixed basis:
\begin{equation}\label{eq1}
\begin{array}{c||c|c|c}
& V_1 & V_2 & V_3 \\
\hline\hline
\underline{e}& 1 & 1 & 1 \\
\hline
\underline{s}& 0 & 2 & 2 \\
\hline
\underline{sts}& 0 & 1-\sqrt{5} & 1+\sqrt{5} 
\end{array}
\end{equation}
Cells in $A$ have exactly the same combinatorics as for $\cQ_5$.
Recall, from \cite[Subsection~5.1]{KM2}, that a subquotient of a transitive based $A$-module 
is called {\em special} provided that it contains a unique maximal 
(in the absolute value) eigenvalue for the element $\underline{e}+\underline{s}+\underline{sts}$.
Special simple $A$-modules are in bijection with two-sided cells, see \cite[Subsection~9.4]{KM2}.
For the cell $\mathcal{L}_e$, the special $A$-module is $V_1$; for the cell $\mathcal{L}_s$, 
the special $A$-module is $V_3$. 

\subsection{Reduction to the rank two case}\label{s4.3}

Let $\hat{\cQ}_5$ be the quotient of $\cQ_5$ by the $2$-ideal generated by $1$-morphisms in $\mathcal{J}$.
Then the only surviving indecomposable $1$-morphism in $\hat{\cQ}_5$ is the identity $1$-morphism, up to
isomorphism. Therefore each simple transitive $2$-representation of $\hat{\cQ}_5$ is a cell 
$2$-representation by \cite[Theorem~18]{MM5}.

Let $\mathbf{M}$ be a simple transitive $2$-representation of $\cQ_5$. 
Recall that $\cQ_5$ is $\mathcal{J}$-simple. Therefore, if $\mathbf{M}$ is not faithful,
it factors through $\hat{\cQ}_5$. Consequently, from the previous paragraph we obtain that 
$\mathbf{M}$ is a cell $2$-representation. Therefore, from now on in this section 
we may assume that $\mathbf{M}$ is faithful.

\begin{lemma}\label{lem12}
Each faithful simple transitive $2$-representation  $\mathbf{M}$ of $\cQ_5$ has rank two
and decategorifies to $V_2\oplus V_3$.
\end{lemma}

\begin{proof}
Consider the transitive module $V=\mathbb{C}\otimes_{\mathbb{Z}}[\mathbf{M}(\mathtt{i})]$ over the
positively based algebra $A$.  As $\mathbf{M}$ is faithful, either $V_2$ or $V_3$ must appear as a
subquotient of $V$. Being special, $V_3$ appears with multiplicity one by
\cite[Section~9]{KM2}. Since $V_2$ is not special, it follows that 
\begin{displaymath}
V\cong  V_3\oplus V_2^{\oplus x}\oplus V_1^{\oplus y}.
\end{displaymath}
As $\underline{sts}$ must have integral trace, from \eqref{eq1} it follows that $x=1$.
Further, transitivity of $\mathbf{M}$ implies that the matrix $M$ of the action of 
$\underline{s}+\underline{sts}$ on $V$ (in the basis of indecomposable projectives)
is a positive integral matrix. Let $N$ be the matrix of the action of $\underline{s}$
(in the same basis).

The $2$-category $\cA_5$ fits into the general framework of \cite{MM5}.
Hence any simple transitive $2$-representation of $\cA_5$ is a cell $2$-representation
by \cite[Theorem~18]{MM5}. This means that $\cA_5$ has two simple transitive $2$-representations and the
matrices $\Lparen \theta_s\Rparen$ in these $2$-representations are $(2)$ and $(0)$, cf.
\cite[Subsection~6.3]{Zi} and, also, Subsection~\ref{s3.4}.

If the matrix $N$ is not upper-triangular after some permutation of the basis, 
then there is a simple transitive $2$-representation of $\cA_5$ which is at least of rank $2$, 
a contradiction. Given that $N$ is upper-triangular, the diagonals have to be $2$ or $0$ by 
the previous paragraph. 
As the $1$-morphism $\theta_s$ is self-adjoint and  $\theta_s^2\cong\theta_s\oplus\theta_s$, 
for each simple object $L$ in $\overline{\mathbf{M}}(\mathtt{i})$, we have either 
$\theta_s L=0$ or $\theta_s L$ has $L$ in the top, cf. Corollary~\ref{cornew01}.
Assume that $N$ has a zero element on the diagonal. From \cite[Lemma~6.4]{Zi} 
(see also Corollary~\ref{cornew01}) it follows that 
$\theta_s L=0$ for some simple object $L\in\overline{\mathbf{M}}(\mathtt{i})$. Then, on the one hand, 
$\theta_{sts}\theta_s L=0$, but, on the other hand, $\theta_{sts}\theta_s\cong \theta_{sts}\oplus \theta_{sts}$.
Therefore $\theta_{sts} L=0$. This implies that the matrix $M$
has a zero row and thus is not positive, a contradiction. Therefore all diagonal elements in
$N$ are equal to $2$. Taking \eqref{eq1} into account, it follows that  $y=0$. 
\end{proof}

\subsection{Combinatorial restrictions on  matrices}\label{s4.4}

Let $\mathbf{M}$ be a faithful simple transitive $2$-representation of $\cQ_5$, 
$P_1$ and $P_2$ be two non-isomorphic indecomposable objects in  ${\mathbf{M}}(\mathtt{i})$ and
$L_1$ and $L_2$ the corresponding simple objects in $\overline{\mathbf{M}}(\mathtt{i})$.
Consider the transitive module $V=\mathbb{C}\otimes_{\mathbb{Z}}[\mathbf{M}(\mathtt{i})]$ over the
positively based algebra $A$. In this section we determine what the combinatorial possibilities are
for the matrices $M_s:=\Lparen \theta_s\Rparen$ and $M_{sts}:=\Lparen \theta_{sts}\Rparen$.

\begin{lemma}\label{lem14}
Up to swapping $P_1$ and $P_2$, the possibilities for the pair $(M_s,M_{sts})$ are:
\begin{enumerate}[$($a$)$]
\item\label{lem14.1} $\left(\left(\begin{array}{cc}2&0\\0&2\end{array}\right),
\left(\begin{array}{cc}1&1\\5&1\end{array}\right)\right)$,
\item\label{lem14.2} $\left(\left(\begin{array}{cc}2&0\\0&2\end{array}\right),
\left(\begin{array}{cc}0&4\\1&2\end{array}\right)\right)$,
\item\label{lem14.3} $\left(\left(\begin{array}{cc}2&0\\0&2\end{array}\right),
\left(\begin{array}{cc}0&2\\2&2\end{array}\right)\right)$,
\item\label{lem14.4} $\left(\left(\begin{array}{cc}2&0\\0&2\end{array}\right),
\left(\begin{array}{cc}0&1\\4&2\end{array}\right)\right)$.
\end{enumerate}
\end{lemma}

\begin{proof}
Similarly to \cite[Lemma~6.4]{Zi}, $M_s$ is equal to twice the identity matrix
(it is a rank-two square matrix with $2$'s on the diagonal satisfying $M_s^2=2M_s$). 
So we only need to  determine $M_{sts}$. Write
\begin{displaymath}
M_{sts}=\left(\begin{array}{cc}a&b\\c&d\end{array}\right).
\end{displaymath}
Then $\underline{sts}\cdot\underline{sts}=2\underline{sts}+2\underline{s}$, given by \eqref{eq1}, 
is equivalent to the following system of equations:
\begin{equation}\label{eq2}
\left\{\begin{array}{lcl}
a^2+bc&=& 2a+4\\
(a+d)b&=& 2b\\
(a+d)c&=& 2c\\
cb+d^2&=& 2d+4
\end{array}\right.
\end{equation}
As $M_s+M_{sts}$ is positive, we have $c\neq 0$ and $b\neq 0$. This implies $a+d=2$. Hence, up to swapping of 
$P_1$ and $P_2$, we have $(a,d)=(1,1)$ or $(a,d)=(0,2)$.

If $(a,d)=(1,1)$, then \eqref{eq2} is equivalent to $bc=5$. Up to swapping of 
$P_1$ and $P_2$, this gives \eqref{lem14.1}. If $(a,d)=(0,2)$, then \eqref{eq2} is equivalent to $bc=4$. 
Note that in this case we cannot swap $P_1$ and $P_2$ anymore, as this will affect the pair $(a,d)$. Therefore
we get all the remaining cases \eqref{lem14.2}, \eqref{lem14.3} and \eqref{lem14.4}.
\end{proof}

\subsection{Ruling out the case of Lemma~\ref{lem14}\eqref{lem14.1}}\label{s4.5}

Here we continue to work in the setup of the previous subsection.

\begin{lemma}\label{lem15}
For a faithful simple transitive $2$-representation of $\cQ_5$, the case of 
Lemma~\ref{lem14}\eqref{lem14.1} is not possible.
\end{lemma}

\begin{proof}
Consider the additive closure of $\theta_s\, L_1$ and $\theta_{sts}\, L_1$. It is stable under 
the action of $\cQ_5$. The usual argument for $\theta_s$ 
(cf. Corollary~\ref{cornew01}) implies that, for $i=1,2$, the module $\theta_s\, L_i$
has length two with simple top and socle isomorphic to $L_i$
(see, for example, proof of \cite[Proposition~22]{MM5} or Subsection~\ref{s3.4}). 
Because of the form of the matrix
$\llbracket \theta_{sts}\rrbracket$ which is transposed to $M_{sts}$, the module $\theta_{sts}\, L_1$ is 
either $L_1\oplus L_2$ or is uniserial with simple top $L_2$ and socle $L_1$, or vice versa.
Applying $\theta_s$, we, on the one hand, should get $\theta_{sts}\, L_1$ doubled, as 
$\theta_s\theta_{sts}=\theta_{sts}\oplus \theta_{sts}$. On the other hand, the resulting module
must contain a self-extension of both $L_1$ and $L_2$, a contradiction.
\end{proof}

\subsection{Ruling out the case of Lemma~\ref{lem14}\eqref{lem14.4}}\label{s4.6}

Here we continue to work in the setup of Subsection~\ref{s4.4}.

\begin{lemma}\label{lem16}
For a faithful simple transitive $2$-representation of $\cQ_5$, the case of 
Lemma~\ref{lem14}\eqref{lem14.4} is not possible.
\end{lemma}

\begin{proof}
Consider the additive closure of $\theta_s\, L_1$ and $\theta_{sts}\, L_1$. It is stable under 
the action of $\cQ_5$. The usual argument for $\theta_s$ (cf. Corollary~\ref{cornew01})
implies that $\theta_s\, L_i$
has length two with simple top and socle isomorphic to $L_i$, for $i=1,2$. Because of the form of the matrix
$\llbracket \theta_{sts}\rrbracket$, the module $\theta_{sts}\, L_1$ is 
$L_2$. Applying $\theta_s$, we, on the one hand, should get $L_2\oplus L_2$, as 
$\theta_s\theta_{sts}=\theta_{sts}\oplus \theta_{sts}$. On the other hand, the resulting module
$\theta_s L_2$ is indecomposable as it has simple top, a contradiction.
\end{proof}

\subsection{Ruling out the case of Lemma~\ref{lem14}\eqref{lem14.2}}\label{s4.7}

Here we continue to work in the setup of Subsection~\ref{s4.4}.

\begin{lemma}\label{lem17}
For a faithful simple transitive $2$-representation of $\cQ_5$, the case of 
Lemma~\ref{lem14}\eqref{lem14.2} is not possible.
\end{lemma}

\begin{proof}
Consider the additive closure of $\theta_s\, L_2$ and $\theta_{sts}\, L_2$. It is stable under 
the action of $\cQ_5$. As usual, for $i=1,2$, the object $\theta_s\, L_i$
has length two with simple top and socle isomorphic to $L_i$. Because of the form of the matrix
$\llbracket \theta_{sts}\rrbracket$, the module $\theta_{sts}\, L_2$ has, as subquotients,
$L_1$ (with multiplicity one) and $L_2$ (with multiplicity two). Similarly to the proof of
Lemma~\ref{lem15}, the module $\theta_{sts}\, L_2$ cannot be semi-simple. 

As $L_1$ has multiplicity one in $\theta_{sts}\, L_2$ and all other simple subquotients are
isomorphic to $L_2$, either socle or top of $\theta_{sts}\, L_2$ contains $L_2$. By adjunction,
we have
\begin{displaymath}
\mathrm{Hom}(\theta_{sts}\, L_2,L_2)\cong \mathrm{Hom}(L_2,\theta_{sts}\, L_2), 
\end{displaymath}
which implies that $L_2$ must appear in both, top and socle of $\theta_{sts}\, L_2$. Now,
$\theta_{sts}\, L_1$ has only $L_2$ as composition subquotients. Therefore, by adjunction,
\begin{displaymath}
\mathrm{Hom}(\theta_{sts}\, L_2,L_1)\cong \mathrm{Hom}(L_2,\theta_{sts}\, L_1)\neq 0 
\end{displaymath}
and therefore $L_1$ is in the top of $\theta_{sts}\, L_2$. A similar argument gives that 
$L_1$ is in the socle of $\theta_{sts}\, L_2$. As  $L_1$ has multiplicity one in $\theta_{sts}\, L_2$,
it must be a direct summand. Now, applying $\theta_s$, we, on the one hand, should get $L_1\oplus L_1$, as 
$\theta_s\theta_{sts}=\theta_{sts}\oplus \theta_{sts}$. On the other hand,  in the resulting module,
the two subquotients $L_1$ must be glued into an indecomposable direct summand  $\theta_s L_1$, a contradiction.
\end{proof}

\subsection{Both $\theta_s$ and $\theta_{sts}$ map simples to projectives}\label{s4.8}

Here we continue to work in the setup of Subsection~\ref{s4.4}.

\begin{lemma}\label{lem19}
For any simple object $L\in \overline{\mathbf{M}}(\mathtt{i})$, both 
$\theta_s\, L$ and  $\theta_{sts}\, L$ are projective objects.
\end{lemma}

\begin{proof}
Combining Lemmata~\ref{lem15}, \ref{lem16} and \ref{lem17}, we know that the pair $(M_s,M_{sts})$
is given by Lemma~\ref{lem14}\eqref{lem14.3}. Let 
\begin{displaymath}
\mathbf{A}:\qquad P_1^{\oplus a_1}\oplus P_2^{\oplus a_2}\overset{\alpha}{\longrightarrow}
P_1^{\oplus b_1}\oplus P_2^{\oplus b_2}
\end{displaymath}

be a minimal projective presentation of $\theta_s\, L$ and 
\begin{displaymath}
\mathbf{B}:\qquad P_1^{\oplus c_1}\oplus P_2^{\oplus c_2}\overset{\beta}{\longrightarrow}
P_1^{\oplus d_1}\oplus P_2^{\oplus d_2}
\end{displaymath}
be a minimal projective presentation of $\theta_{sts}\, L$. As $\theta_s$ is self-adjoint, it maps
projective resolutions to projective resolutions. From \eqref{eq0} and the explicit matrix for
$M_{s}$ we obtain that $\theta_s$ doubles both  $\theta_s$, $\theta_{sts}$ and all projective modules.
Hence $\theta_s$ sends $\mathbf{A}$ to $\mathbf{A}\oplus \mathbf{A}$
and $\mathbf{B}$ to $\mathbf{B}\oplus \mathbf{B}$.

As $\theta_{sts}$ is self-adjoint, it maps projective resolutions to projective resolutions
(however, we do not know whether it sends minimal resolutions to minimal resolutions). 
Applying $\theta_{sts}$ to $\mathbf{A}$ and $\mathbf{B}$ and using \eqref{eq0} and the explicit form of
$M_{sts}$ produces the following two systems of inequalities:
\begin{displaymath}
\left\{\begin{array}{lcl}
2b_2&\geq & 2d_1\\
2b_1+2b_2&\geq & 2d_2\\
2d_2&\geq & 2d_1+2b_1\\
2d_1+2d_2&\geq & 2d_2+2b_2
\end{array}\right.
\qquad
\left\{\begin{array}{lcl}
2a_2&\geq & 2c_1\\
2a_1+2a_2&\geq & 2c_2\\
2c_2&\geq & 2c_1+2a_1\\
2c_1+2c_2&\geq & 2c_2+2a_2
\end{array}\right.
\end{displaymath}
These inequalities imply the equalities $d_1=b_2$, $d_2=b_1+b_2$, $c_1=a_2$ and also $c_2=a_1+a_2$. 
Therefore $\theta_{sts}$ maps $\mathbf{A}$ to $\mathbf{B}\oplus \mathbf{B}$ and
$\mathbf{B}$ to $\mathbf{A}\oplus \mathbf{A}\oplus \mathbf{B}\oplus \mathbf{B}$.

The above proves that the ideal generated by $\alpha$ and $\beta$ is stable under the action 
of $\cQ_5$. Since $\mathbf{M}$ is simple transitive, this ideal therefore has to be zero, 
so $\alpha=\beta=0$, cf. \cite[Lemma~12]{MM5}. This completes the proof.
\end{proof}

\subsection{Proof of Theorem~\ref{thm11}}\label{s4.9}

Here we continue to work under the assumptions of Subsection~\ref{s4.4}.
Combining Lemmata~\ref{lem15}, \ref{lem16} and \ref{lem17}, we know that the pair 
$(M_s,M_{sts})$ is given by Lemma~\ref{lem14}\eqref{lem14.3}. Consider the additive 
closure of $\theta_s\, L_1$ and $\theta_{sts}\, L_1$. Both these objects are
projective and the additive closure  is stable under the action of $\cQ_5$.
The module $\theta_s\, L_1$ is an indecomposable module of length two with simple
top and socle isomorphic to $L_1$, by the usual argument (cf. Corollary~\ref{cornew01}). Therefore
$\theta_s\, L_1\cong P_1$. 

The module $\theta_{sts}\, L_1$ has length two and cannot be semisimple by the argument
in the proof of Lemma~\ref{lem15}. Therefore it is an indecomposable module of 
length two with simple top and socle isomorphic to $L_2$. Therefore
$\theta_{sts}\, L_1\cong P_2$. This implies that the Cartan matrix of $\mathbf{M}$ is
\begin{equation}\label{eq4}
\left(\begin{array}{cc}2&0\\0&2\end{array}\right).
\end{equation}
Mapping $\mathbbm{1}_{\mathtt{i}}$ to $L_1$ extends to a strict $2$-natural transformation 
$\Phi:\mathbf{P}_{\mathtt{i}}\to\mathbf{M}$. 

Consider now the cell $2$-representation $\mathbf{C}_{\mathcal{L}_s}$. It is a
faithful simple transitive $2$-representation of $\cQ_5$. Hence the above 
shows that \eqref{eq4} is also the Cartan matrix of $\mathbf{C}_{\mathcal{L}_s}$.
From the uniqueness of a maximal left ideal in the construction of cell $2$-representations, 
it follows that $\Phi$ factors through $\mathbf{C}_{\mathcal{L}_s}$. Comparing the
Cartan matrices, we see that the induced $2$-natural transformation from 
$\mathbf{C}_{\mathcal{L}_s}$ to $\mathbf{M}$ is an equivalence.

\section{Simple transitive  $2$-representations of $\cQ_4$}\label{s5}

\subsection{The result}\label{s5.1}

Our main result for the $2$-category $\cQ_4$ is the following:

\begin{theorem}\label{thm21}
The $2$-category $\cQ_4$ has three equivalence classes of simple transitive $2$-representations,
namely, the cell $2$-representations $\mathbf{C}_{\mathcal{L}_e}$ and  $\mathbf{C}_{\mathcal{L}_s}$
together with the simple transitive $2$-representation $\mathbf{N}$ constructed in Subsection~\ref{s5.3}.
\end{theorem}

\subsection{The algebra $[\cQ_4](\mathtt{i},\mathtt{i})$}\label{s5.2}

We identify the ring $R=[\cQ_4](\mathtt{i},\mathtt{i})$ with the corresponding subquotient of 
$\mathbb{Z}D_{2\cdot 4}$. With this identification, the ring $R$ has basis 
$\{\underline{e},\underline{s},\underline{sts}\}$ and the following multiplication table
of $x\cdot y$:
\begin{equation}\label{eq10}
\begin{array}{c||c|c|c}
x\setminus y& \underline{e} & \underline{s} & \underline{sts} \\
\hline\hline
\underline{e}& \underline{e} & \underline{s} & \underline{sts} \\
\hline
\underline{s}& \underline{s} & 2\underline{s} & 2\underline{sts} \\
\hline
\underline{sts}& \underline{sts} & 2\underline{sts} & 2\underline{s} 
\end{array}
\end{equation}

The $\mathbb{C}$-algebra $A:=\mathbb{C}\otimes_{\mathbb{Z}}R$ is commutative and split semisimple.
The three $1$-dimensional representations of $A$ are given by the following table which describes
the action of the basis elements in any fixed basis:
\begin{equation}\label{eq11}
\begin{array}{c||c|c|c}
& V_1 & V_2 & V_3 \\
\hline\hline
\underline{e}& 1 & 1 & 1 \\
\hline
\underline{s}& 0 & 2 & 2 \\
\hline
\underline{sts}& 0 & -2 & 2 
\end{array}
\end{equation}
Cells in $A$ have exactly the same combinatorics as for $\cQ_4$. 
For the cell $\mathcal{L}_e$,
the special $A$-module, in the sense of \cite{KM2}, 
is $V_1$; for the cell $\mathcal{L}_s$, the special $A$-module is $V_3$.

\subsection{Reduction to ranks one and  two}\label{s5.4}

Let $\hat{\cQ}_4$ be the quotient of $\cQ_4$ by the $2$-ideal generated by $1$-morphisms in $\mathcal{J}$.
Then the only surviving indecomposable $1$-morphism in $\hat{\cQ}_5$ is the identity $1$-morphism, up to
isomorphism. Therefore each simple transitive $2$-representation of $\hat{\cQ}_4$ is a cell 
$2$-representation by \cite[Theorem~18]{MM5}.

Let $\mathbf{M}$ be a simple transitive $2$-representation of $\cQ_4$. 
Recall that $\cQ_4$ is $\mathcal{J}$-simple. Therefore, if $\mathbf{M}$ is not faithful,
it factors through $\hat{\cQ}_4$. Consequently, from the previous paragraph we obtain that 
$\mathbf{M}$ is a cell $2$-representation. Therefore, from now on in this section 
we may assume that $\mathbf{M}$ is faithful.

\begin{lemma}\label{lem22}
A faithful simple transitive $2$-representation  $\mathbf{M}$ of $\cQ_4$ either has rank two
and decategorifies to $V_2\oplus V_3$ or it has rank one and decategorifies to $V_3$.
\end{lemma}

\begin{proof}
Consider the transitive module 
$V=\mathbb{C}\otimes_{\mathbb{Z}}[\mathbf{M}(\mathtt{i})]$ over the
positively based algebra $A$, see Subsection~\ref{s2.3}.
As $\mathbf{M}$ is faithful, either $V_2$ or $V_3$ must appear as a
subquotient of $V$. As $V_3$ is special, it has multiplicity one.
Since $V_2$ is not special, it follows that 
\begin{displaymath}
V\cong  V_3\oplus V_2^{\oplus x}\oplus V_1^{\oplus y}.
\end{displaymath}
Similarly to the proof of Lemma~\ref{lem12} one shows that $y=0$.
As $\underline{sts}$ must have non-negative trace, we have $x\leq 1$. The claim follows.
\end{proof}

\subsection{Combinatorial restrictions on  matrices}\label{s5.5}

Let $\mathbf{M}$ be a faithful simple transitive $2$-representation of $\cQ_4$. Let 
$P_1$ or $P_1$ and $P_2$ be non-isomorphic indecomposable objects in  ${\mathbf{M}}(\mathtt{i})$ and
$L_1$ or  $L_1$ and $L_2$ the corresponding simple objects in $\overline{\mathbf{M}}(\mathtt{i})$.
Consider the transitive module $V=\mathbb{C}\otimes_{\mathbb{Z}}[\mathbf{M}(\mathtt{i})]$ over the
positively based algebra $A$. In this section we determine what the combinatorial possibilities for
the matrices $M_s:=\Lparen \theta_s\Rparen$ and $M_{sts}:=\Lparen \theta_s\Rparen$ are.

\begin{lemma}\label{lem23}
Up to swapping $P_1$ and $P_2$, the only possibilities for the pair $(M_s,M_{sts})$ are:
\begin{enumerate}[$($a$)$]
\item\label{lem23.1} $\left((2),(2)\right)$,
\item\label{lem23.2} $\left(\left(\begin{array}{cc}2&0\\0&2\end{array}\right),
\left(\begin{array}{cc}0&2\\2&0\end{array}\right)\right)$,
\item\label{lem23.3} $\left(\left(\begin{array}{cc}2&0\\0&2\end{array}\right),
\left(\begin{array}{cc}0&4\\1&0\end{array}\right)\right)$.
\end{enumerate}
\end{lemma}

\begin{proof}
If $\mathbf{M}$ has rank one, the possibility \eqref{lem23.1} follows combining Lemma~\ref{lem22}
and \eqref{eq11}. 

If $\mathbf{M}$ has rank two, then, as in Lemma~\ref{lem14}, $M_s$ is equal to twice the identity 
matrix, so we only need to  determine $M_{sts}$. Write
\begin{displaymath}
M_{sts}=\left(\begin{array}{cc}a&b\\c&d\end{array}\right).
\end{displaymath}
Then $\underline{sts}\cdot\underline{sts}=2\underline{s}$, given by \eqref{eq10}, 
is equivalent to the following system of equations:
\begin{displaymath}
\left\{\begin{array}{lcl}
a^2+bc&=& 4\\
(a+d)b&=& 0\\
(a+d)c&=& 0\\
cb+d^2&=& 4
\end{array}\right.
\end{displaymath}
As $M_s+M_{sts}$ is positive, we have $c\neq 0$ and $b\neq 0$. This implies $a+d=0$ and thus $a=0=d$ 
and $bc=4$. Up to swapping of  $P_1$ and $P_2$, this gives possibilities 
\eqref{lem23.2} and \eqref{lem23.3}.
\end{proof}

\subsection{Ruling out the case of Lemma~\ref{lem23}\eqref{lem23.3}}\label{s5.6}

Here we continue to work in the setup of the previous subsection.

\begin{lemma}\label{lem24}
For a faithful simple transitive $2$-representation of $\cQ_4$, the case of 
Lemma~\ref{lem23}\eqref{lem23.3} is not possible.
\end{lemma}

\begin{proof}
Consider the additive closure of $\theta_s\, L_2$ and $\theta_{sts}\, L_2$. It is stable under 
the action of $\cQ_4$. The usual argument for $\theta_s$ (cf. Corollary~\ref{cornew01}) implies that $\theta_s\, L_i$
has length two with simple top and socle isomorphic to $L_i$, for $i=1,2$. Because of the form of the matrix
$\llbracket \theta_{sts}\rrbracket$ which is transposed to $M_{sts}$, we have
$\theta_{sts}\, L_2\cong L_1$.  Now, on the one hand,  $\theta_s\theta_{sts}\, L_2\cong L_1\oplus L_1$, as 
$\theta_s\theta_{sts}=\theta_{sts}\oplus \theta_{sts}$. On the other hand, the 
usual argument for $\theta_s$ (cf. Corollary~\ref{cornew01}) implies that $\theta_s\, L_1$ 
has length two with simple top and socle isomorphic to $L_1$, a contradiction.
\end{proof}

\subsection{Both $\theta_s$ and $\theta_{sts}$ map simples to projectives}\label{s5.7}

Here we continue to work in the setup of Subsection~\ref{s5.5}.

\begin{lemma}\label{lem25}
For any simple object $L\in \overline{\mathbf{M}}(\mathtt{i})$, both 
$\theta_s\, L$ and  $\theta_{sts}\, L$ are projective objects.
\end{lemma}

\begin{proof}
Combining Lemmata~\ref{lem23} and \ref{lem24}, we know that the pair $(M_s,M_{sts})$
is given by Lemma~\ref{lem23}\eqref{lem23.1} or by Lemma~\ref{lem23}\eqref{lem23.2}. 
In the rank one case given by Lemma~\ref{lem23}\eqref{lem23.1} it is clear that both
$\theta_s$ and $\theta_{sts}$ double both $\theta_s\, L$, $\theta_{sts}\, L$ and 
all projectives. Therefore they send minimal projective resolutions to minimal
projective resolutions and the claim follows by similar arguments as in Lemma~\ref{lem19}.

In the rank two case given by Lemma~\ref{lem23}\eqref{lem23.2}, 
$\theta_s$ doubles everything; while $\theta_{sts}$ doubles and swaps indices.
Again, it follows easily that both $\theta_s$ and $\theta_{sts}$ send minimal 
projective resolutions to minimal projective resolutions 
and the claim follows by similar arguments as in Lemma~\ref{lem19}.
\end{proof}

\subsection{Some evidence for the existence of an additional simple transitive 
$2$-rep\-re\-sent\-ation of $\cQ_4$}\label{s5.35}

The ring $[\cQ_6](\mathtt{i},\mathtt{i})$ has basis 
$\{\underline{e},\underline{s},\underline{sts},\underline{ststs}\}$ 
and the following multiplication table of $x\cdot y$:
\begin{displaymath}
\begin{array}{c||c|c|c|c}
x\setminus y& \underline{e} & \underline{s} & \underline{sts} & \underline{ststs}\\
\hline\hline
\underline{e}& \underline{e} & \underline{s} & \underline{sts} & \underline{ststs}\\
\hline
\underline{s}& \underline{s} & 2\underline{s} & 2\underline{sts} & 2\underline{ststs}\\
\hline
\underline{sts}& \underline{sts} & 2\underline{sts} & 2\underline{s}+
2\underline{sts}+2\underline{ststs}& 2\underline{sts}\\ 
\hline
\underline{ststs}& \underline{ststs} & 2\underline{ststs} & 2\underline{sts}& 2\underline{s}\\ 
\end{array}
\end{displaymath}
Comparing this with \eqref{eq10} suggests the possibility of a connection between $\cQ_4$
and the $2$-subcategory $\widetilde{\cQ}_6$ of $\cQ_6$ generated by $\theta_s$ and $\theta_{ststs}$
(however, establishing such a connection explicitly might be very hard).

Now we can consider the restriction of the cell $2$-representation $\mathbf{C}_{\mathcal{L}_s}$
of $\cQ_6$ to $\widetilde{\cQ}_6$. This restriction is not transitive, so we can consider its 
weak Jordan-H{\"o}lder series in the sense of \cite[Section~4]{MM5}. From  the above
multiplication table, we see that  $\underline{ststs}\cdot \underline{sts}=2\underline{sts}$.
This means that there is a simple transitive subquotient in this weak Jordan-H{\"o}lder series,
for which the pair $(\Lparen \theta_s\Rparen,\Lparen \theta_{ststs}\Rparen)$ of matrices 
has the form $((2),(2))$. This suggests that a similar thing should also exist for $\cQ_4$.

\subsection{An additional simple transitive $2$-representation of $\cQ_4$}\label{s5.3}

The aim of this subsection is to construct a simple transitive $2$-representation of $\cQ_4$
which is not equivalent to a cell $2$-representation. This is done in several steps:
\begin{itemize}
\item We start with the cell $2$-representation $\mathbf{Q}$ of $\cQ_4$ which
corresponds to the cell $\mathcal{J}$. The underlying category of this $2$-representation 
has two isomorphism classes of indecomposable objects. We observe that there is 
a $2$-representations $\mathbf{K}$ of $\cQ_4$ which is equivalent to $\mathbf{Q}$
and which has a non-trivial automorphism swapping  the indecomposable objects. 
\item Next, we modify $\mathbf{K}$ to another equivalent $2$-representation 
$\mathbf{L}$ which has the same kind of automorphism that, additionally, 
is a strict involution, i.e. squares to the identity functor.
At the intermediate stage we modify $\mathbf{K}$ to another equivalent 
$2$-representation which has a coherent $\mathbb{Z}/2\mathbb{Z}$-action
(we use the terminology of \cite[Pages~135-136]{Se}).
This involves a subtle chase of certain modifications for 
the cell $2$-representation, which is rather technical and occupies a major part of this subsection.
\item Finally, we use the orbit category construction 
(see \cite{CM} and Subsection~\ref{s5.9} for details) to produce a
``quotient'' of $\mathbf{L}$ which turns out to be a simple transitive $2$-representation of $\cQ_4$
whose underlying category has one isomorphism class of indecomposable objects.
\end{itemize}

So, let us now do the work.
Consider the cell $2$-representation $\mathbf{Q}:=\mathbf{C}_{\mathcal{L}_s}$ of $\cQ_4$
and its abelianization $\overline{\mathbf{Q}}$. Let $P_s$ and $P_{sts}$ be representatives of
the isomorphism classes of the indecomposable projective objects in $\overline{\mathbf{Q}}(\mathtt{i})$
and $L_s$ and $L_{sts}$ be the respective simple tops. Let $B$ denote the basic underlying algebra of 
$\overline{\mathbf{Q}}(\mathtt{i})$ and $f_s$ and $f_{sts}$ denote pairwise orthogonal 
primitive idempotents of $B$ corresponding to $P_s$ and $P_{sts}$, respectively. 
The matrices of the action of $\theta_s$ and $\theta_{sts}$ are given by Lemma~\ref{lem23}\eqref{lem23.2}. As $\mathbf{Q}$ is simple transitive, 
both $\theta_s$ and $\theta_{sts}$ send simple objects in $\overline{\mathbf{Q}}(\mathtt{i})$ to
projective objects in $\overline{\mathbf{Q}}(\mathtt{i})$, by Lemma~\ref{lem25}. As $\cQ_4$ is fiat, from 
\cite[Lemma~13]{MM5} it follows that both $\theta_s$ and $\theta_{sts}$ act as projective
endofunctors of $\overline{\mathbf{Q}}(\mathtt{i})$. 

Recall that $D$ is the algebra of dual numbers introduced in Subsection~\ref{s3.3}.
Taking into account the results of Subsection~\ref{s3.4}, 
from the matrix $M_s$ we can deduce that $B\cong D\oplus D$  and 
$f_s+f_{sts}=1$. The action of $\theta_s$ is given by tensoring with the $B$-$B$--bimodule
$(Bf_{s}\otimes f_{s}B)\oplus(Bf_{sts}\otimes f_{sts}B)$ while
the action of $\theta_{sts}$ is, similarly,  given by tensoring with the $B$-$B$--bimodule
$(Bf_{sts}\otimes f_{s}B)\oplus(Bf_{s}\otimes f_{sts}B)$. 

Our first intermediate goal is to construct a strict $2$-natural automorphism of $\mathbf{Q}$
which swaps the isomorphism classes of indecomposable objects in $\mathbf{Q}(\mathtt{i})$. 
We cannot do it directly and instead construct a different, but equivalent, 
$2$-representation $\mathbf{K}$ where such a strict $2$-natural automorphism
is easy to find.

Set $\mathbf{Q}^{(0)}:=\mathbf{Q}$ and let $\mathbf{Q}^{(1)}$ denote the $2$-representation of $\cQ_4$
given by the action of $\cQ_4$ on the category of projective objects in 
$\overline{\mathbf{Q}}^{(0)}(\mathtt{i})$. Recursively, for $k\geq 1$, define 
$\mathbf{Q}^{(k)}$ as the $2$-representation of $\cQ_4$ given by the action of $\cQ_4$ 
on the category of projective objects in $\overline{\mathbf{Q}}^{(k-1)}(\mathtt{i})$.
For every $k\geq 0$, we have a strict $2$-natural transformation 
$\Lambda_k:{\mathbf{Q}}^{(k-1)}\to \mathbf{Q}^{(k)}$ which sends an object $X$ to the diagram
$0\to X$ and a morphism $\alpha:X\to X'$ to the diagram 
\begin{displaymath}
\xymatrix{
0\ar[rr]\ar[d] && X\ar[d]^{\alpha}\\
0\ar[rr] && X'.
} 
\end{displaymath}
Clearly, each such $\Lambda_k$ is an equivalence. 

Denote by $\mathbf{K}$ the inductive limit of the directed system
\begin{equation}\label{eq5}
\mathbf{Q}^{(0)}\overset{\Lambda_0}{\longrightarrow}
\mathbf{Q}^{(1)}\overset{\Lambda_1}{\longrightarrow}
\mathbf{Q}^{(2)}\overset{\Lambda_2}{\longrightarrow}\dots. 
\end{equation}
Then $\mathbf{K}$ is a $2$-representation of $\cQ_4$ which is equivalent to $\mathbf{Q}$.

\begin{lemma}\label{lem31}
There is a strict $2$-natural transformation $\Psi:\mathbf{K}\to \mathbf{K}$ which is an equivalence 
and which swaps the isomorphism classes of the indecomposable projective objects. 
\end{lemma}

\begin{proof}
Consider the unique strict $2$-natural transformation $\Phi:\mathbf{P}_{\mathtt{i}}\to \overline{\mathbf{Q}}$
which sends $\mathbb{1}_{\mathtt{i}}$ to $L_{sts}$. By the above discussion, both 
$\theta_s\, L_{sts}$ and $\theta_{sts}\, L_{sts}$ are indecomposable objects in 
$\mathbf{Q}^{(1)}(\mathtt{i})$. Using similar arguments as in Subsection~\ref{s4.9},
it follows that $\Phi$ factors through $\mathbf{C}_{\mathcal{L}_s}$ and, therefore,  gives rise to a strict
equivalence $\Phi^{(0)}:\mathbf{Q}^{(0)}\to \mathbf{Q}^{(1)}$. Applying abelianization, for
every $k\geq 0$, we obtain a strict equivalence $\Phi^{(k)}:\mathbf{Q}^{(k)}\to \mathbf{Q}^{(k+1)}$,
which is, by construction, compatible with \eqref{eq5}. Now we can take $\Psi$ as the
inductive limit of $\Phi^{(k)}$.
\end{proof}

Consider a new finitary $2$-representation $\mathbf{K}'$ of $\cQ_4$ defined as follows:
\begin{itemize}
\item Objects of $\mathbf{K}'(\mathtt{i})$ are sequences 
$(X_n,\alpha_n)_{n\in\mathbb{Z}}$, where $X_n$ is an object in 
$\mathbf{K}(\mathtt{i})$ and $\alpha_n\colon \Psi(X_n)\to X_{n+1}$ an isomorphism 
in $\mathbf{K}(\mathtt{i})$, for all $n\in\mathbb{Z}$. 
\item Morphisms in $\mathbf{K}(\mathtt{i})'$ from $(X_n,\alpha_n)_{n\in\mathbb{Z}}$
to $(Y_n,\beta_n)_{n\in\mathbb{Z}}$ are sequences of morphisms $f_n\colon X_n\to Y_n$ 
in $\mathbf{K}(\mathtt{i})$ such that   
\begin{displaymath}
\xymatrix{
\Psi(X_n)\ar[rr]^{\alpha_n}\ar[d]_{\Psi(f_n)}&&X_{n+1}\ar[d]^{f_{n+1}}\\
\Psi(Y_n)\ar[rr]^{\beta_n}&&Y_{n+1}\\
}
\end{displaymath}
commutes for all $n\in\mathbb{Z}$. 
\item The action of $\cQ_4$ on $\mathbf{K}'(\mathtt{i})$ is inherited from the
action of $\cQ_4$ on $\mathbf{K}(\mathtt{i})$ component-wise.
\end{itemize}
The construction of $\mathbf{K}'(\mathtt{i})$ from $\mathbf{K}(\mathtt{i})$ is the
standard construction which turns a category with an autoequivalence (in our case
$\Psi$) into an equivalent category with an automorphism, cf. \cite{Ke,BL}. 

We have the
strict $2$-natural transformation $\Pi:\mathbf{K}'\to \mathbf{K}$ given by
projection onto the zero component of a sequence. This $\Pi$ is an equivalence, by construction.
We also have a strict $2$-natural transformation $\Psi':\mathbf{K}'\to \mathbf{K}'$
given by shifting the entries of the sequences by one, that is, sending 
$(X_n,\alpha_n)_{n\in\mathbb{Z}}$ to $(X_{n+1},\alpha_{n+1})_{n\in\mathbb{Z}}$,
with the similar obvious action on morphisms. Note that 
$\Psi':\mathbf{K}'(\mathtt{i})\to \mathbf{K}'(\mathtt{i})$ is an 
automorphism. Similarly to $\Psi$, the functor $\Psi'$
swaps the isomorphism classes of indecomposable objects in $\mathbf{K}'(\mathtt{i})$.
As $\Psi^2$ is isomorphic to the identity functor on 
$\mathbf{K}(\mathtt{i})$, it follows, by construction, that $(\Psi')^2$ is isomorphic to 
the identity functor on $\mathbf{K}'(\mathtt{i})$. However, we need the following stronger statement.

\begin{lemma}\label{lem-modif}
Let $\mathrm{Id}:\mathbf{K}'\to \mathbf{K}'$ denote the identity 
$2$-natural transformation.
\begin{enumerate}[$($i$)$]
\item\label{lem-modif.1} 
There is an invertible modification $\eta:\mathrm{Id}\to (\Psi')^2$.
\item\label{lem-modif.2}
For any $\eta$ as in \eqref{lem-modif.1}, we have
$\mathrm{id}_{(\Psi')^2}\circ_0 \eta=\eta\circ_0\mathrm{id}_{(\Psi')^2}$.
\item\label{lem-modif.3}
For any $\eta'\in\mathrm{Hom}_{\ccQ_4\text{-}\mathrm{mod}}(\mathrm{Id},(\Psi')^2)$, we have
$\mathrm{id}_{(\Psi')^2}\circ_0 \eta'=\eta'\circ_0\mathrm{id}_{(\Psi')^2}$.
\end{enumerate}
\end{lemma}

\begin{proof}
Let $X_1$ and  $X_2$ be representatives of the different isomorphism classes of
indecomposable objects in $\mathbf{K}'(\mathtt{i})$. Fix an  isomorphism
$\eta_{X_1}:X_1\to (\Psi')^2(X_1)$. Note that, after the abelianization of $\mathbf{K}'$,
the object $X_1$ becomes isomorphic to $\theta_s\, L_1$, where  $L_1$ is the simple top
of the indecomposable projective object $0\to X_1$. As 
$\theta_{sts}\circ\theta_s\cong \theta_{sts}\oplus\theta_{sts}$,
and $(\Psi')^2$ is a strict $2$-natural transformation,
it follows that the morphism $\theta_{sts}(\eta_{X_1})$ can be written in the form
\begin{displaymath}
\left(\begin{array}{cc}\alpha&0\\0&\alpha\end{array}\right), 
\end{displaymath}
for some isomorphism $\alpha:\theta_{sts}\, L_1\to (\Psi')^2(\theta_{sts}\, L_1)$. We may now define
$\eta_{X_2}$ as the map induced from $\alpha$ via a fixed isomorphism $X_2\cong \theta_{sts}\, L_1$. 
This uniquely determines a bijective natural transformation  $\eta:\mathrm{Id}\to (\Psi')^2$.

As $\theta_{sts}\circ\theta_{sts}\cong \theta_s\oplus \theta_s$, from the above construction
it follows that 
\begin{equation}\label{neq8}
\mathrm{id}_{\theta_{sts}}\circ_0\eta=\eta\circ_0 \mathrm{id}_{\theta_{sts}}.
\end{equation}
Note that the $2$-category $\cQ_4$ is monoidally generated by $\theta_{sts}$. Therefore \eqref{neq8} implies
that the $\eta$ constructed above is, in fact, a modification. Claim~\eqref{lem-modif.1} follows.
 
Because of the naturality of $\eta$, we have the commutative diagram
\begin{displaymath}
\xymatrix{ 
(\Psi')^{-2}\ar[rr]^{\eta\circ_0\mathrm{id}_{(\Psi')^{-2}}}
\ar[d]_{\eta\circ_0\mathrm{id}_{(\Psi')^{-2}}}&&
\mathrm{Id}\ar[d]^{\mathrm{id}_{(\Psi')^2}\circ_0\eta\circ_0\mathrm{id}_{(\Psi')^{-2}}}\\
\mathrm{Id}\ar[rr]^{\eta}&&(\Psi')^2
} 
\end{displaymath}
As $\eta\circ_0\mathrm{id}_{(\Psi')^{-2}}$ is invertible, 
claim~\eqref{lem-modif.2} follows.
As invertible elements inside the space
$\mathrm{Hom}_{\ccQ_4\text{-}\mathrm{mod}}(\mathrm{Id},(\Psi')^2)$
exist by claim~\eqref{lem-modif.1}, they also span this space
(by the usual argument that invertible modifications form
a Zariski open set which is non-empty by claim~\eqref{lem-modif.1}). 
Therefore claim~\eqref{lem-modif.3} follows from claim~\eqref{lem-modif.2}
by linearity.
\end{proof}

The next statement basically says that there is a coherent action of the group 
$\mathbb{Z}/2\mathbb{Z}$ on $\mathbf{K}'$.

\begin{proposition}\label{prop-particularPhi}
There is an invertible modification $\eta:\mathrm{Id}\to (\Psi')^2$ 
for which we have the equality $\mathrm{id}_{\Psi'}\circ_0 \eta=\eta\circ_0\mathrm{id}_{\Psi'}$.
\end{proposition}

\begin{proof}
Consider the space $\mathrm{Hom}(\mathrm{Id},(\Psi')^2)$ of all natural transformations
from $\mathrm{Id}$ to $(\Psi')^2$, as endofunctors of $\mathbf{K}'(\mathtt{i})$.
As $\mathrm{Id}$ and $(\Psi')^2$ are isomorphic, the space $\mathrm{Hom}(\mathrm{Id},(\Psi')^2)$
is isomorphic, by fixing any isomorphism between $\mathrm{Id}$ and $(\Psi')^2$, to the space
$\mathrm{End}(\mathrm{Id})$ of natural endomorphisms of the identity functor on 
$\mathbf{K}'(\mathtt{i})$. We have the usual isomorphism $\mathrm{End}(\mathrm{Id})\cong B$,
since $B$ is commutative. In particular, $\mathrm{End}(\mathrm{Id})$ has dimension four.

The vector space $\mathrm{Hom}_{\ccQ_4\text{-}\mathrm{mod}}(\mathrm{Id},(\Psi')^2)$ of
modifications is a subspace in the space $\mathrm{Hom}(\mathrm{Id},(\Psi')^2)$.
An element $\zeta\in \mathrm{Hom}(\mathrm{Id},(\Psi')^2)$ is uniquely determined by its values
on any pair $X_1$ and $X_2$ of representatives of the two isomorphism classes of 
indecomposable objects in  $\mathbf{K}'(\mathtt{i})$. The $2$-category $\cQ_4$ is monoidally 
generated by $\theta_{sts}$. Therefore, as $\theta_{sts}\,X_1\cong X_2\oplus X_2$,
the axiom $\zeta\circ_0 \mathrm{id}_{\theta_{sts}}=\mathrm{id}_{\theta_{sts}}\circ_0\zeta$ for
modifications implies that an element $\zeta\in \mathrm{Hom}_{\ccQ_4\text{-}\mathrm{mod}}(\mathrm{Id},(\Psi')^2)$
is uniquely determined by its value on $X_1$. Consequently, 
$\mathrm{Hom}_{\ccQ_4\text{-}\mathrm{mod}}(\mathrm{Id},(\Psi')^2)$ has dimension two.

The map 
\begin{equation}\label{eqnn37}
\zeta\mapsto \mathrm{id}_{\Psi'}\circ_0 \zeta\circ_0 \mathrm{id}_{{\Psi'}^{-1}}  
\end{equation}
is an invertible linear transformation of the space $\mathrm{Hom}(\mathrm{Id},(\Psi')^2)$ which preserves
$\mathrm{Hom}_{\ccQ_4\text{-}\mathrm{mod}}(\mathrm{Id},(\Psi')^2)$. We denote this transformation by  $T$. From
Lemma~\ref{lem-modif}\eqref{lem-modif.3} we know that $T^2$ is the identity map, when restricted to 
$\mathrm{Hom}_{\ccQ_4\text{-}\mathrm{mod}}(\mathrm{Id},(\Psi')^2)$. 
Consequently, $T$ is diagonalizable (in the restriction) with eigenvalues $1$ and $-1$.

The algebra $D$ is positively graded in the usual way (the degree of $x$ is two), which also
makes $B$ into a positively graded algebra. The bimodules which represent the actions of
$\theta_s$ and $\theta_{sts}$ are gradeable and all our constructions above are gradeable
as well. Consequently, $\mathrm{Hom}(\mathrm{Id},(\Psi')^2)$ inherits a grading form $B$
and $\mathrm{Hom}_{\ccQ_4\text{-}\mathrm{mod}}(\mathrm{Id},(\Psi')^2)$ is a graded subspace of 
$\mathrm{Hom}(\mathrm{Id},(\Psi')^2)$ isomorphic to $D$ (as a graded space). The map
\eqref{eqnn37} is homogeneous of degree zero, which implies that both homogeneous elements 
in $\mathrm{Hom}_{\ccQ_4\text{-}\mathrm{mod}}(\mathrm{Id},(\Psi')^2)$ are eigenvectors for
the linear map \eqref{eqnn37}.

Let $\eta\in \mathrm{Hom}_{\ccQ_4\text{-}\mathrm{mod}}(\mathrm{Id},(\Psi')^2)$ be 
a non-zero homogeneous element of degree zero. Then $\eta$ is, obviously, invertible.
By the above, $\eta$ is an eigenvector for the map \eqref{eqnn37}. To complete
the proof, we only need to check  that the corresponding eigenvalue is $1$ and not $-1$.

To check the eigenvalue, we will use a description of both $\eta$ and $\Psi'$ in terms of 
$B$-$B$--bimodules which is explicit enough to see that the transformation \eqref{eqnn37},
when applied to $\eta$, cannot introduce any signs. For this we
need to recall the construction of $\Psi'$ which is based on the
construction of $\Psi$ (cf. \cite[Subsection~4.6]{MM1}). 
Let $\alpha_1,\alpha_2,\dots,\alpha_k$ be a basis in the linear space of 
degree two  $2$-endomorphisms of $\theta_{sts}$. Then the  module $L_{sts}$ in 
$\overline{\mathbf{Q}}(\mathtt{i})$, which was used to define $\Psi$, has a presentation of the form 
\begin{displaymath}
\theta_{sts}^{\oplus k}\overset{\alpha}{\longrightarrow}\theta_{sts}, 
\end{displaymath}
where $\alpha=(\alpha_1,\alpha_2,\dots,\alpha_k)$. For objects $\theta\in \mathbf{C}_{\mathcal{L}_s}$,
the functor $\Psi$ is given by 
\begin{displaymath}
\theta\quad\mapsto\qquad \theta\theta_{sts}^{\oplus k}\overset{\theta(\alpha)}{\longrightarrow}\theta\theta_{sts},
\end{displaymath}
by construction. The $2$-representation $\mathbf{K}'$ is $2$-faithful 
because of $\mathcal{J}$-simplicity of $\cQ_5$, see Subsection~\ref{s3.2}.
Therefore, using the connection to $B$-mod as explained in the beginning of this subsection, 
the above can be written in terms of explicit $B$-$B$--bimodules and bimodule maps which 
represent the action of $\theta_s$ and $\theta_{sts}$. Let us denote by $Q_s$ and $Q_{sts}$
the $B$-$B$-bimodules which represent the actions of $\theta_s$ and $\theta_{sts}$, respectively.

We can also consider a similar presentation
\begin{displaymath}
\theta_{s}^{\oplus m}\overset{\beta}{\longrightarrow}\theta_{s}, 
\end{displaymath}
for $L_S$, where $\beta=(\beta_1,\beta_2,\dots,\beta_m)$ with $\beta_1,\beta_2,\dots,\beta_m$ being a basis in 
the linear space of  degree two  $2$-endomorphisms of $\theta_{s}$. Then the functor 
\begin{displaymath}
\theta\quad\mapsto\qquad \theta\theta_{s}^{\oplus m}\overset{\theta(\beta)}{\longrightarrow}\theta\theta_{s}
\end{displaymath}
is isomorphic to the identity functor on the cell $2$-representation via the isomorphism induced by the
multiplication map $\mu:B\otimes_{\mathbb{C}}B\to B$ (note that $Q_s$ is a direct summand of 
$B\otimes_{\mathbb{C}}B$). Note that $\mu$ does not introduce any additional signs to its arguments.

We have the usual isomorphism $Q_{sts}\otimes_B Q_{sts}\cong Q_s\oplus Q_s\langle-2\rangle$, where 
$\langle-2\rangle$ denotes the degree shift by $2$ in the positive direction (i.e. with the top concentrated
in degree two). Composition of the projection $Q_{sts}\otimes_B Q_{sts}\tto Q_s$ with the 
multiplication map $\mu$ from the previous paragraph gives rise
to a natural transformation from $\Psi^2$ to the identity functor.
In degree zero, the projection $Q_{sts}\otimes_B Q_{sts}\tto Q_s$ is given in terms of 
contracting to $\mathbb{C}$ some middle tensor factors, surrounded by $\otimes_{\mathbb{C}}$, 
of a tensor monomial. Note that this does not introduce any additional signs. 
By construction, the induced $2$-natural transformation is a homogeneous modification of degree zero, 
so we can choose $\eta$ such that it corresponds to this map.

We can take $\theta_s$ and $\theta_{sts}$ as representatives of isomorphism
classes of indecomposable objects in the cell $2$-representation. The modification $\eta$
is uniquely determined by its values on any of these objects.
We have to check that $\Psi(\eta)=\eta_{\Psi}$, while we already know that 
$\Psi(\eta)=\pm\eta_{\Psi}$. To do this, we can explicitly write down the bimodules representing the
evaluations of $\Psi(\eta)$ and $\eta_{\Psi}$ at some object, say $\theta_s$. This gives rather big complexes
of bimodules, in which we just need to compare the degree zero parts in position zero. 
We already know that the degree zero parts are either equal or differ by a sign. 
However, as explained above, neither the multiplication map nor the projection 
$Q_{sts}\otimes_B Q_{sts}\tto Q_s$ introduce any new signs in our picture. As we work over $\mathbb{C}$,
that is in characteristic zero, we cannot introduce any signs by adding and multiplying positive elements.
This implies that, indeed, $\Psi(\eta)=\eta_{\Psi}$. The statement of the proposition follows.
\end{proof}

From now on we fix some invertible modification $\eta:\mathrm{Id}\to (\Psi')^2$
as given by Proposition~\ref{prop-particularPhi}. Now we use $\eta$ to replace $\mathbf{K}'$
by an equivalent $2$-representation $\mathbf{L}$ in which there is an analogue of $\Psi'$ 
(see the definition of $\Theta$ below) that
is, additionally, a strict involution.
Define a small category $\mathbf{L}(\mathtt{i})$ as follows:
\begin{itemize}
\item objects  in $\mathbf{L}(\mathtt{i})$  are all $4$-tuples $(X,Y,\alpha,\beta)$, where we have
$X,Y\in \mathbf{K}'(\mathtt{i})$, while $\alpha:X\to \Psi'(Y)$ and $\beta:Y\to \Psi'(X)$ are 
isomorphisms such that the following conditions are satisfied
\begin{equation}\label{eqnn9}
\left\{ 
\begin{array}{rcl} 
\eta_{Y}^{-1}\circ_1 {\Psi'}(\alpha)\circ_1\beta&=&\mathrm{id}_Y,\\
\beta\circ_1\eta_{Y}^{-1}\circ_1 {\Psi'}(\alpha)&=&\mathrm{id}_{\Psi'(X)},\\ 
\eta_{X}^{-1}\circ_1{\Psi'}(\beta)\circ_1\alpha&=&\mathrm{id}_X,\\
\alpha\circ_1\eta_{X}^{-1}\circ_1{\Psi'}(\beta)&=&\mathrm{id}_{\Psi'(Y)}. 
\end{array}
\right.
\end{equation}
\item morphisms in $\mathbf{L}(\mathtt{i})$ from $(X,Y,\alpha,\beta)$ to $(X',Y',\alpha',\beta')$
are pairs $(\zeta,\xi)$, where $\zeta:X\to X'$ and $\xi:Y\to Y'$ are morphisms in 
$\mathbf{K}'(\mathtt{i})$ such that the diagrams
\begin{displaymath}
\xymatrix{
X\ar[rr]^{\alpha}\ar[d]_{\zeta}&&\Psi(Y)\ar[d]^{\Psi'(\xi)}\\
X'\ar[rr]^{\alpha'}&&\Psi'(Y')\\
} \quad\text{ and }\quad
\xymatrix{
Y\ar[rr]^{\beta}\ar[d]_{\xi}&&\Psi(X)\ar[d]^{\Psi'(\zeta)}\\
Y'\ar[rr]^{\beta'}&&\Psi'(X')\\
} 
\end{displaymath}
commute;
\item the composition and identity morphisms are the obvious ones.
\end{itemize}
The category $\mathbf{L}(\mathtt{i})$ comes equipped with an action of $\cQ_4$, defined component-wise,
using the action of $\cQ_4$ on $\mathbf{K}(\mathtt{i})$. This is well-defined as the
$2$-natural transformation $\Psi$ is strict by Lemma~\ref{lem31} and, moreover, $\eta$ is a modification. 
We denote the corresponding $2$-representation of $\cQ_4$ by $\mathbf{L}(\mathtt{i})$. 

\begin{lemma}\label{lemnequiv}
Restriction to the first component of a quadruple defines a 
strict $2$-natural transformation  $\Upsilon:\mathbf{L}\to \mathbf{K}'$.
This $\Upsilon$ is an equivalence. 
\end{lemma}

\begin{proof}
That the restriction in question is a strict $2$-natural transformation is clear by construction.
We need to check that it is an equivalence. For this we need to check two things. The first one is
the fact that $\beta$ is uniquely determined by $\alpha$.
and that $\xi$ is uniquely determined by  $\zeta$. 
This follows easily from the definitions and the equations in \eqref{eqnn9}. 
The second thing to check is that, given $X$, $Y$ and $\alpha$,
there is a $\beta$ such that \eqref{eqnn9} is satisfied.

By assumptions, we have $X\cong \Psi'(Y)$. The first two equations in \eqref{eqnn9} just say that 
$\beta$ and $\eta_{Y}^{-1}\circ_1 {\Psi'}(\alpha)$ are mutual inverses in 
$\mathrm{Hom}_{\mathbf{L}(\mathtt{i})}(X,\Psi'(Y))$
and $\mathrm{Hom}_{\mathbf{L}(\mathtt{i})}(\Psi'(Y),X)$. Therefore these two equations
describe equivalent conditions on $\alpha$ and $\beta$. Similarly, the two last equations in 
\eqref{eqnn9} describe equivalent conditions on $\alpha$ and $\beta$. It remains to check that
the first equation is equivalent to the last one. For this we apply $\Psi'$ to the first equation
and compare the outcome with the last equation, taking into account that $\Psi'(\beta)$ is an
isomorphism. We get
\begin{equation}\label{eqnn10}
\Psi'(\eta_{Y}^{-1})\circ_1 {\Psi'}^2(\alpha)=\alpha\circ_1\eta_{X}^{-1}. 
\end{equation}
As $\Psi'(\eta_{Y})=\eta_{\Psi'(Y)}$, by our choice of $\eta$ and Proposition~\ref{prop-particularPhi},
the equality in \eqref{eqnn10} holds due to naturality of $\eta$. The claim of the lemma follows.
\end{proof}

Define an endofunctor $\Theta$ on $\mathbf{L}(\mathtt{i})$ by sending $(X,Y,\alpha,\beta)$ to 
$(Y,X,\beta,\alpha)$ with the obvious action on morphisms. From all symmetries in  the 
definition of $\mathbf{L}(\mathtt{i})$ it follows that $\Theta$ is a strict involution 
and it also strictly commutes with the action of $\cQ_4$.

Finally, consider the category $\mathbf{N}(\mathtt{i})$ defined as follows:
\begin{itemize}
\item $\mathbf{N}(\mathtt{i})$ has the same objects as $\mathbf{L}(\mathtt{i})$,
\item morphisms in  $\mathbf{N}(\mathtt{i})$ are defined, for objects $X,Y\in\mathbf{L}(\mathtt{i})$, via
\begin{displaymath}
\mathrm{Hom}_{\mathbf{N}(\mathtt{i})}(X,Y):=
\mathrm{Hom}_{\mathbf{L}(\mathtt{i})}(X,Y)\oplus  \mathrm{Hom}_{\mathbf{L}(\mathtt{i})}(X,\Theta(Y)),
\end{displaymath}
\item composition and identity morphisms in $\mathbf{N}(\mathtt{i})$ are induced from those
in $\mathbf{L}(\mathtt{i})$ in the obvious way, see \cite[Definition~2.3]{CM} for details.
\end{itemize}

\begin{proposition}\label{prop27}
{\hspace{2mm}}

\begin{enumerate}[$($i$)$]
\item\label{prop27.1} The category  $\mathbf{N}(\mathtt{i})$ is a finitary $\mathbb{C}$-linear category.
\item\label{prop27.2} The category  $\mathbf{N}(\mathtt{i})$ is equipped with an 
action of $\cQ_4$ induced from that on $\mathbf{L}(\mathtt{i})$.
\item\label{prop27.3} The  obvious functor
$\Xi:\mathbf{L}(\mathtt{i})\to \mathbf{N}(\mathtt{i})$ is a strict $2$-natural transformation.
\end{enumerate}
\end{proposition}

\begin{proof}
It is well-known that the category $\mathbf{N}(\mathtt{i})$ is an additive 
$\mathbb{C}$-linear category, see \cite[Page~552]{Ke} and \cite[Section~2]{CM}. 
For each indecomposable object $X$ in $\mathbf{L}(\mathtt{i})$, the object $\Theta(X)$ is
not isomorphic to $X$, moreover, $\mathrm{Hom}_{\mathbf{L}(\mathtt{i})}(X,\Theta(X))=0$. 
From the construction it is now easy to see that that the endomorphism algebra of $X$
in $\mathbf{N}(\mathtt{i})$ is the same as in $\mathbf{L}(\mathtt{i})$, in particular, it is
local. Therefore $X$ is indecomposable in $\mathbf{N}(\mathtt{i})$ and thus the fact that 
$\mathbf{N}(\mathtt{i})$ is idempotent split and Krull-Schmidt with finitely many indecomposable
objects follows from the corresponding properties in $\mathbf{L}(\mathtt{i})$,
by additivity and $\mathbb{C}$-linearity. This proves claim~\eqref{prop27.1}. 

Claims \eqref{prop27.2} and \eqref{prop27.3} follow directly by construction, 
taking into account the fact that the $2$-natural transformation $\Theta$ is strict.
\end{proof}

Proposition~\ref{prop27} gives us the $2$-representation $\mathbf{N}$ of $\cQ_4$. This 
$2$-representation has rank one. Being of rank one, this $2$-representation is transitive. 
As the underlying algebra of $\mathbf{L}(\mathtt{i})$ is $D\oplus D$, from the 
proof of Proposition~\ref{prop27} we see that the underlying algebra of
$\mathbf{N}(\mathtt{i})$ is $D$. Therefore the restriction of $\mathbf{N}$ to the
$2$-subcategory $\cA_4$ is simple transitive. This means that $\mathbf{N}$ itself is simple transitive as well.

\subsection{Some abstract nonsense on orbit categories}\label{s5.9}

Let $G$ be an abelian group and $\mathcal{Q}$ a category equipped with a (strict) $G$-action,
$g\mapsto \mathrm{F}_g:\mathcal{Q}\to \mathcal{Q}$.
Following \cite[Definition~2.3]{CM} (see also \cite{As}), we define the {\em skew category} $\mathcal{Q}[G]$
as follows:
\begin{itemize}
\item $\mathcal{Q}[G]$ has the same objects as $\mathcal{Q}$;
\item $\displaystyle \mathcal{Q}[G](i,j)=\bigoplus_{g\in G}\mathcal{Q}(i,\mathrm{F}_g(j))$;
\item composition in $\mathcal{Q}[G]$ is given by composition in $\mathcal{Q}$, after adjustment.
\end{itemize}
We denote by $\mathrm{T}:\mathcal{Q}\to\mathcal{Q}[G]$ the natural inclusion.
Note that, in Subsection~\ref{s5.3}, we have
$\mathbf{N}(\mathtt{i})=\mathbf{L}(\mathtt{i})[G]$, where $G=\{\mathrm{Id}_{\mathbf{L}(\mathtt{i})},\Theta\}$.

\begin{lemma}\label{lemnn57}
For any $g\in G$, there exists a natural isomorphism 
$\xi(g):\mathrm{T}\cong \mathrm{T}\circ \mathrm{F}_g$ 
of functors from $\mathcal{Q}$ to $\mathcal{Q}[G]$
such that $\xi(e)$ is the identity and the 
following diagram commutes, for all $g,h\in G$:
\begin{equation}\label{eqnn11}
\xymatrix{ 
\mathrm{T}\ar[rr]^{\xi(gh)}\ar[d]_{\xi(h)} && \mathrm{T}\circ \mathrm{F}_{gh}\ar@{=}[d]\\
\mathrm{T}\circ \mathrm{F}_{h}\ar[rr]^{\xi(g)\circ_0\mathrm{id}_{\mathrm{F}_{h}}} && 
\mathrm{T}\circ \mathrm{F}_{g}\circ \mathrm{F}_{h}
}
\end{equation}
\end{lemma}

\begin{proof}
On any object $i\in \mathcal{Q}$, the value of $\xi(g)$ is given by the identity 
morphism $\varepsilon_i\in \mathcal{Q}(i,i)$ which is, at the same time, 
an element in $\mathcal{Q}[G](i,\mathrm{F}_g(i))$. It is clear from the construction that 
$\xi(g)$ defined in this way is a natural isomorphism, that 
$\xi(e)$ is the identity and that \eqref{eqnn11} commutes.
\end{proof}

The following statement is similar to \cite[Proposition~2.6]{As}.

\begin{proposition}\label{propnn58}
Let $\mathcal{Q}'$ be a category and $\mathrm{K}:\mathcal{Q}\to \mathcal{Q}'$ be a functor.
Assume  that, for each $g\in G$, there is an isomorphism $\tau(g):\mathrm{K}\to \mathrm{K}\circ \mathrm{F}_g$ 
such that $\tau(e)$ is the identity and the following diagram commutes, for all $g,h\in G$:
\begin{equation}\label{eqnn58}
\xymatrix{ 
\mathrm{K}\ar[rr]^{\tau(gh)}\ar[d]_{\tau(h)} && \mathrm{K}\circ \mathrm{F}_{gh}\ar@{=}[d]\\
\mathrm{K}\circ \mathrm{F}_{h}\ar[rr]^{\tau(g)\circ_0\mathrm{id}_{\mathrm{F}_{h}}} && 
\mathrm{K}\circ \mathrm{F}_{g}\circ \mathrm{F}_{h}
}
\end{equation}
Then there is a functor $\overline{\mathrm{K}}:\mathcal{Q}[G]\to \mathcal{Q}'$
such that $\mathrm{K}=\overline{\mathrm{K}}\circ\mathrm{T}$.
\end{proposition}

\begin{proof}
On objects and morphisms from $\mathcal{Q}(i,j)$, we define $\overline{\mathrm{K}}$
as $\mathrm{K}$. This, in particular, guarantees $\mathrm{K}=\overline{\mathrm{K}}\circ\mathrm{T}$.
It remains to define $\overline{\mathrm{K}}$ on morphisms from $\mathcal{Q}(i,\mathrm{F}_g(j))$,
considered as a subset of $\mathcal{Q}[G](i,j)$,
where $g\in G$ is different from the identity element. 

For every $\alpha\in \mathcal{Q}(i,\mathrm{F}_g(j))\subset \mathcal{Q}[G](i,j)$, we define
$\overline{\mathrm{K}}(\alpha):=\tau(g)_{j}^{-1}\circ {\mathrm{K}}(\alpha)$
(this, in particular, agrees with the previous paragraph in the case  $g=e$). 
Now we have to check the functoriality of this construction.

Let $\alpha\in \mathcal{Q}(i,\mathrm{F}_g(j))\subset \mathcal{Q}[G](i,j)$ and 
$\beta\in \mathcal{Q}(j,\mathrm{F}_h(k))\subset \mathcal{Q}[G](i,k)$,
where $g,h\in G$. Then we compute:
\begin{displaymath}
\begin{array}{rcll}
\overline{\mathrm{K}}(\beta)\circ \overline{\mathrm{K}}(\alpha)&=&
\tau(h)_{k}^{-1}\circ {\mathrm{K}}(\beta)\circ 
\tau(g)_{j}^{-1}\circ {\mathrm{K}}(\alpha)& \text{(definition)}\\ &=&
\tau(h)_{k}^{-1}\circ \tau(g)_{\mathrm{F}_h(k)}^{-1}\circ 
\mathrm{K}({\mathrm{F}_g}(\beta))\circ {\mathrm{K}}(\alpha)& \text{(naturality of $\tau(g)$)}\\&=&
\tau(h)_{k}^{-1}\circ \tau(g)_{\mathrm{F}_h(k)}^{-1}\circ 
\mathrm{K}({\mathrm{F}_g}(\beta)\circ \alpha)& \text{(functoriality of $\mathrm{K}$)}\\&=&
\tau(gh)_{k}^{-1}\circ 
{\mathrm{K}}(\mathrm{F}_g(\beta)\circ \alpha)& \text{(by \eqref{eqnn58}).}\\ 
\end{array}
\end{displaymath}
This completes the proof.
\end{proof}

\subsection{Proof of Theorem~\ref{thm21}}\label{s5.8}

In the rank two case, the combinatorics of the action is described by 
Lemma~\ref{lem23}\eqref{lem23.2}. Consider the additive closure of 
$\theta_s\, L_1$ and $\theta_{sts}\, L_1$. It is stable under 
the action of $\cQ_4$. By Lemma~\ref{lem25}, both $\theta_s\, L_1$ and $\theta_{sts}\, L_1$
are projective modules. A usual argument involving the action of  $\theta_s$ 
(cf. Corollary~\ref{cornew01}) shows that they
are indecomposable. Therefore $P_1\cong \theta_s\, L_1$ and $P_2\cong \theta_{sts}\, L_1$.
Now the fact that $\mathbf{M}$ is isomorphic to $\mathbf{C}_{\mathcal{L}_s}$ is proved using
arguments similar to the ones in Subsection~\ref{s4.9}.

Let now $\mathbf{M}$ be  a simple transitive $2$-representation of $\cQ_4$ of rank one.
Consider its abelianization $\overline{\mathbf{M}}$ and let $\Phi$ be the unique strict 
$2$-natural transformation from $\mathbf{P}_{\mathtt{i}}$ to $\overline{\mathbf{M}}$ which
sends $\mathbb{1}_{\mathtt{i}}$ to a (unique up to isomorphism) simple object $L$ in 
$\overline{\mathbf{M}}(\mathtt{i})$. In the same way as we did several times above, 
$\Phi$ restricts to a strict $2$-natural transformation from $\mathbf{C}_{\mathcal{L}_s}$
to $\mathbf{M}^{(1)}$ which sends both $\theta_s$ and $\theta_{sts}$ to indecomposable
(but now isomorphic) objects. 

Applying abelianization to the above, we get a strict $2$-natural 
transformation from $\mathbf{Q}^{(k)}$
to $\mathbf{M}^{(k+1)}$, for each $k\geq 0$, which is compatible with the canonical
inclusions $\mathbf{Q}^{(k)}\hookrightarrow\mathbf{Q}^{(k+1)}$
and $\mathbf{M}^{(k+1)}\hookrightarrow\mathbf{M}^{(k+2)}$. Taking the limit, gives a 
strict $2$-natural transformation $\Gamma$ from the $2$-representation $\mathbf{K}$ from 
Subsection~\ref{s5.3} to the limit $\hat{\mathbf{M}}$ of the inductive system
\begin{displaymath}
\mathbf{M}^{(0)}{\longrightarrow}
\mathbf{M}^{(1)}{\longrightarrow}
\mathbf{M}^{(2)}{\longrightarrow}\dots. 
\end{displaymath}
Note that $\mathbf{K}$ is equivalent to $\mathbf{C}_{\mathcal{L}_s}$ while 
$\hat{\mathbf{M}}$ is equivalent to $\mathbf{M}$. 
This gives us the solid part of the diagram
\begin{equation}\label{eqnn59}
\xymatrix{
\mathbf{L}\ar[rr]^{\Pi\circ \Upsilon}\ar[d]_{\Xi}&&\mathbf{K}\ar[d]^{\Gamma}\\
\mathbf{N}\ar@{..>}[rr]_{\Delta}&&\hat{\mathbf{M}}.
}
\end{equation}

Let $X_1$ be an indecomposable object in $\mathbf{L}(\mathtt{i})$.
Then $\Theta(X_1)$ is also indecomposable and not isomorphic to $X_1$.
Together $X_1$ and $\Theta(X_1)$ generate $\mathbf{L}(\mathtt{i})$ as an additive category.
From the construction above we see that the objects 
$\Gamma\circ\Pi\circ\Upsilon(X_1)$ and $\Gamma\circ\Pi\circ\Upsilon\circ \Theta(X_1)$
are isomorphic in $\hat{\mathbf{M}}(\mathtt{i})$. 
Fixing an isomorphism between these two object extends to an invertible modification 
$\nu: \Gamma\circ\Pi\circ\Upsilon\to \Gamma\circ\Pi\circ\Upsilon\circ \Theta$.
As $\Theta$ is a strict involution, it follows that the modification $\nu$, together with the identity modification 
on $\Gamma\circ\Pi\circ\Upsilon$ satisfy all conditions of Proposition~\ref{propnn58}.
In particular, the following diagram establishes \eqref{eqnn58} in this case:
\begin{displaymath}
\xymatrix{ 
\Gamma\circ\Pi\circ\Upsilon\ar[d]_{\nu}\ar@/^/[rr]^{\nu} && \Gamma\circ\Pi\circ\Upsilon\circ \Theta
\ar@/^/[ll]^{\nu^{-1}}\ar[d]^{\nu^{-1}}\\
\Gamma\circ\Pi\circ\Upsilon\circ \Theta\ar@/^/[rr]^{\nu^{-1}} && 
\Gamma\circ\Pi\circ\Upsilon\circ \Theta\circ \Theta\ar@/^/[ll]^{\nu}.\\
}
\end{displaymath}

From Proposition~\ref{propnn58}, we thus get the dotted functor
$\Delta$ in \eqref{eqnn59} which makes \eqref{eqnn59} a strictly commutative diagram. 
As $\Theta$ strictly commutes with the action of $\cQ_4$, from the construction in 
Proposition~\ref{propnn58} we see that $\Delta$ is, in fact, a strict $2$-natural 
transformation. By construction,
$\Delta$ maps a (unique up to isomorphism) indecomposable generator, say $X$, of 
$\mathbf{N}(\mathtt{i})$ to a (unique up to isomorphism) 
indecomposable generator, say $Y$, of $\hat{\mathbf{M}}(\mathtt{i})$. From our analysis 
above we know that the endomorphism algebras of both $X$ and $Y$ are isomorphic to $D$.
Restriction of $\Delta$ to the action of $\cA_4$ is, therefore, an equivalence
as both these restricted $2$-representations must be equivalent to the cell $2$-representation of
$\cA_4$. This implies that $\Delta$ is an
equivalence between $\mathbf{N}$ and $\hat{\mathbf{M}}$. Hence 
$\mathbf{N}$ and $\mathbf{M}$ are equivalent as well. The proof is complete.


\noindent
M.~M.: Center for Mathematical Analysis, Geometry, and Dynamical Systems, Departamento de Matem{\'a}tica, 
Instituto Superior T{\'e}cnico, 1049-001 Lisboa, PORTUGAL \& Departamento de Matem{\'a}tica, FCT, 
Universidade do Algarve, Campus de Gambelas, 8005-139 Faro, PORTUGAL, email: {\tt mmackaay\symbol{64}ualg.pt}

\noindent

V.~M.: Department of Mathematics, Uppsala University, Box. 480,
SE-75106, Uppsala, SWEDEN, email: {\tt mazor\symbol{64}math.uu.se}

\end{document}